\documentclass{amsart}

\usepackage{mathrsfs}
\usepackage{amssymb}
\usepackage{microtype}
\usepackage{graphicx}
\usepackage{subcaption}

\usepackage{xcolor}
\usepackage{float}
\usepackage{needspace}
\usepackage[hidelinks]{hyperref}
\hypersetup{pdftitle={Weighted normalized curve shortening flow with applications in pseudo-Euclidean spaces},pdfauthor={Baichuan Hu and Xiang Ma}}
\newtheorem{theorem}{Theorem}[section]
\newtheorem{lemma}[theorem]{Lemma}

\newtheorem{corollary}[theorem]{Corollary}
\newtheorem{proposition}[theorem]{Proposition}
\theoremstyle{definition}
\newtheorem{definition}[theorem]{Definition}
\newtheorem{example}[theorem]{Example}

\theoremstyle{remark}
\newtheorem{remark}[theorem]{Remark}

\numberwithin{equation}{section}

\title[Weighted normalized curve shortening flow]{Weighted normalized curve shortening flow with applications in pseudo-Euclidean spaces}

\author{Baichuan Hu}
\address{School of Mathematical Sciences, Peking University, Beijing 100871, P.R. China}
\curraddr{}
\email{2301110016@pku.edu.cn}
\thanks{}

\author{Xiang Ma}
\address{LMAM, School of Mathematical Sciences, Peking University, Beijing 100871, P.R. China}
\curraddr{}
\email{maxiang@math.pku.edu.cn}
\thanks{Corresponding author: Xiang Ma(maxiang@math.pku.edu.cn); partially supported by NSFC grant 11831005.}

\date{}

\begin{document}

%    Abstract is required.
\begin{abstract}
The focus of this paper is the curve shortening flow for closed spacelike curves in pseudo-Euclidean spaces, which has very few results so far. Will they produce singularities where certain tangent line tends to light cone? If not, will such a curve shrink to a circular point? To answer these questions, we establish a dichotomy for planar weighted normalized curve shortening flow with uniformly positive and bounded weights. Applying to closed smooth spacelike curves in pseudo-Euclidean spaces that admit a one-to-one convex projection onto a spacelike plane, at their finite maximal time we will see: either the curve shrinks to a point and becomes asymptotically circular, or the tangent directions subsequentially approach the null cone. Both alternatives occur. In the first case, this proves our previous conjecture that a strong spacelike curve in $\mathbb{R}^{2,q}$ with index 1 will converge to a circular point under the usual CSF. In the latter case, a monotone area-bivector defect is found in $\mathbb R^{2,1}$, which gives a quantitative obstruction to point collapse. Explicit examples of spacelike curves with lightlike tangent limit are given.
\end{abstract}

\maketitle

\tableofcontents

\section{Introduction and Notations}

The curve shortening flow (CSF) has been studied for several decades. In \cite{GH86}, Gage and Hamilton proved that any strictly convex planar closed curve shrinks to a point and becomes asymptotically circular under the curve shortening flow. In \cite{Gra87}, Grayson proved that any embedded planar curve will become convex before developing any singularities under the CSF.

For space CSF in Euclidean spaces, Huisken's Monotonicity Formula (\cite{Hui90}) and Altschuler's singularity classification result (\cite{Alt91}) are important tools. Recently, Sun showed that if the initial curve has a one-to-one convex projection onto a plane, then it develops a Type \uppercase\expandafter{\romannumeral1} singularity and becomes asymptotically circular under the CSF (\cite{Sun26}). Andrews and Zhou \cite{AZ26} proved round-point convergence for compact spacelike-convex submanifolds with one spacelike codimension. In the curve case relevant here, their result applies to closed strong spacelike curves of index 1 in $\mathbb R^{2,k}$. Our convex-projection hypothesis does not require the osculating planes to be spacelike; which allows peculiar examples with null tangent limit in Section~\ref{sec:examples}.

We got interested in this topic several years ago, mainly because in $\mathbb R^{2,1}$, we wanted to establish a Fenchel-type inequality for the total curvature of a closed strong spacelike curve (which means that its tangent vector and curvature vector span a spacelike plane) with index 1. This inequality
\[\int k ds \le 2\pi\]
has a reversed direction (comparing with the Euclidean version), due to the indefinite ambient metric. Notice that the total curvature increases monotonically under the curve shortening flow in this case. If we can prove that under the flow this curve will shrink to a circular point, the reversed Fenchel inequality follows immediately. Yet the pseudo-Euclidean metric as well as the high codimension turned out to be great challenges to the usual research method. Later we proved this inequality in \cite{YMW16} using a Crofton formula, and provided another proof in \cite{YM19} by considering the spanned maximal disk and applying the Gauss-Bonnet formula. For the above results on the reversed Fenchel inequality and a preliminary study of curve shortening flow in $\mathbb R^{2,1}$, see also \cite{Ye16}.

Although the reversed Fenchel inequality is established, we have been wondering about two basic questions for the CSF in the pseudo-Euclidean case: (1) Must the flow forces a strong spacelike curve with index 1 to become circular? (2) Under weaker assumptions, which kind of spacelike curves could deform to a peculiar limit with a null tangent? (This would be a new kind of singularity besides the usual type-1 and type-2 singularities in the Euclidean context.)

This paper is a final affirmative answer to both questions. The turning point is giving up the original idea that the Lorentz metric should play an invariant and important role in the analysis. On the contrary, in this paper we identify $\mathbb{R}^{p,q}$ with the same linear space with a Euclidean metric $\mathbb{R}^{p+q}$ in a fixed coordinate system. And we weaken our assumption to that the spacelike curve $\Gamma$ projects to a spacelike plane which is a one-to-one correspondence to a convex plane curve $\gamma$. The image $P_{\Sigma}(\Gamma)$ will deform simultaneously as a weighted curve shortening flow. By a geometric way to looking at this flow we can derive the crucial dichotomy result.

Let $\Gamma(\theta, t)$ be a solution to the curve shortening flow (CSF) in $\mathbb{R}^{p,q}$
$$
\Gamma_{ t} = \Gamma_{ss}
$$
and the maximal existence time is $T_\Gamma$. We take
$$
\gamma(\theta,{\sigma}) = [2(T_\Gamma- t)]^{-1/2}\Gamma(\theta,t)
$$
where ${\sigma} = -\frac{\ln(T_\Gamma- t)}{2}+\frac{\ln T_\Gamma}{2}$. Then $\gamma$ satisfies the normalized curve shortening flow equation:
$$
\gamma_{{\sigma}} = \gamma_{{s}{s}}+\gamma.
$$
where ${s}$ denotes the arc-length parameter of $\gamma$.

Let $\Sigma$ be a spacelike plane and assume that the orthogonal projection $\hat\gamma_1=P_\Sigma\gamma$ is regular. Then
$$
(\hat \gamma_1)_{{\sigma}} = (\hat \gamma_1)_{{s}{s}}+\hat \gamma_1 = \left(\frac{d\hat {s}}{d{s}}\right)^2(\hat \gamma_1)_{\hat {s}\hat {s}} + \frac{d^2\hat {s}}{d{s}^2}(\hat \gamma_1)_{\hat {s}}+\hat \gamma_1.
$$
where $\hat {s}$ denotes the arc-length parameter of $\hat \gamma_1$.

Under a proper reparameterization (for details, see \cite{Sun24}, Lemma 3.4), we let $\hat \gamma(\delta,{\sigma}) =\hat  \gamma_1(\theta(\delta,{\sigma}),{\sigma})$ such that
$$
(\hat \gamma)_{{\sigma}} =  \left(\frac{d\hat {s}}{d{s}}\right)^2(\hat \gamma)_{\hat {s}\hat {s}}+\hat \gamma.
$$

In general, if an evolving curve $r$ satisfies
\begin{align}
    r_t = r_{ss}+r,
\end{align}
then we say $r$ is a solution to the \emph{normalized flow} (NF). More generally, if
\begin{align}
    r_t = a(\theta,t)r_{ss}+r\,,a(\theta,t)> 0
\end{align}
then we say $r$ is a solution to the \emph{weighted normalized flow} (WNF). The positive function $a(\theta,t)$ is called the \emph{weight}. The above indicates that if $r$ is a solution to the NF, then for a projection $P$, $P\circ r$ is a solution to the WNF, where the weight function is $\left(\frac{d\hat s}{ds}\right)^2$.

In what follows, we establish a dichotomy of solutions to the WNF, where the weights are positively uniformly bounded. Combining with interior estimates of parabolic equations (Theorem~\ref{Nash and Schauder estimate}), this result is pivotal in the study of singularities of the CSF. \\

\noindent\textbf{Basic notations.}

The pseudo-Euclidean space $\mathbb{R}^{p,q}$ is the $(p+q)$-dimensional real linear space endowed with the standard pseudo-Euclidean metric $\langle\cdot ,\cdot\rangle$ with index $q$. In a fixed unit orthogonal coordinate system, we always let
$$
\langle X,Y\rangle=\sum_{i=1}^{p}x_i y_i-\sum_{j=1}^{q}x_{p+j}y_{p+j},\quad X,Y\in\mathbb R^{p+q}.
$$

In any fixed coordinate system in $\mathbb{R}^{p,q}$(where $p\geq2)$, $xy$-plane denotes the spacelike plane spanned by $(1,0,...,0)$ and $(0,1,0,..,0)$, and $P_{xy}$ denotes the orthogonal projection to $xy$-plane.

$\langle\cdot,\cdot\rangle$ denotes the inner product in pseudo-Euclidean spaces, while $w\cdot v$ denotes the Euclidean inner product of vectors $w,v$.

Suppose $A$ and $B$ are two $n$-dimensional real linear spaces, each with a fixed coordinate system. Then the coordinate copy $Q$ from $A$ to $B$ is defined as
$$
Q((x_1,\ldots,x_n)\in A)=(x_1,\ldots,x_n)\in B.
$$

Let $\Sigma$ be a spacelike $2$-dimensional subspace in $\mathbb{R}^{p,q}$. Then $P_{\Sigma}$ denotes the pseudo-Euclidean orthogonal projection to $\Sigma$. Correspondingly,  $P^*_{\Sigma}$ denotes the Euclidean orthogonal projection to $\Sigma$, where $\Sigma$ be a $2$-dimensional subspace in $\mathbb{R}^{p+q}$.

Suppose there is a $2$-subspace $\Sigma$ in a Euclidean space and a vector $v = v_1+v_2$, such that $v_1\in \Sigma$ and $v_2\in \Sigma^\perp$. Then the angle between $v$ and $\Sigma$ is defined as
$$\angle(v,\Sigma) = \arctan \frac{|v_2|}{|v_1|}\in [0,\frac{\pi}{2}].
$$

Let $r = r(u)$ be a $C^2$ regular parametrized curve. If $r'(u)$ and $r''(u)$ are linearly independent, then we can define the osculating subspace of $r$ at $r(u)$ as:
$$
\mathrm{Osc}_r(u) = \mathrm{span}\{r'(u), r''(u)\}.
$$
The osculating subspaces are independent of the choice of $u$.

Unless otherwise specified, $\frac{\partial}{\partial s}$ always denotes the partial derivative with respect to the arc-length parameter of the corresponding curve.

$B_r$ denotes the open Euclidean ball centered at the origin with radius $r$. We write $|\cdot|_E$ for the Euclidean norm in the fixed coordinates and $|v|=\sqrt{\langle v,v\rangle}$ for a spacelike vector. Boundedness of pseudo-Euclidean unit tangents always refers to $|\Gamma_s|_E$, not to their identically unit pseudo-Euclidean norm. Smooth convergence of curves is with respect to the topology induced by the Euclidean norm, and is understood up to smooth reparameterizations.\\

\noindent\textbf{Curvature and strong spacelike condition.}
We say a $C^2$ spacelike curve $\Gamma$ is \emph{strong spacelike} if $\Gamma$ admits a spacelike osculating subspace at each point. For a strong spacelike curve $\Gamma$, its curvature $k$ is defined as
$$
k = \left|\frac{d^2\Gamma}{ds^2}\right| = \sqrt{\left\langle \frac{d^2\Gamma}{ds^2},\frac{d^2\Gamma}{ds^2} \right\rangle}.
$$

For the definition of their index and the properties of strong spacelike curves, the reader may refer to \cite{YMW16} and \cite{YM19}.

\vspace{1cm}

\noindent\textbf{Main results.}
Firstly in Section 2-4, we establish a dichotomy for planar weighted normalized curve shortening flow with weights uniformly positively bounded.
\begin{theorem}\label{main theorem}
Let $\Gamma:S^1\times[0,+\infty)\to \mathbb{R}^2$ be a smooth mapping such that
\begin{align}\label{WNF}
    \Gamma_{\sigma} = a(\theta,\sigma)\Gamma_{ss}+\Gamma
\end{align}
where
$$
\frac{\partial}{\partial s} = \frac{1}{|\Gamma_\theta|}\frac{\partial}{\partial \theta}
$$
denotes the partial derivative with respect to the arc-length parameter. Also, for any ${\sigma}\in [0,+\infty)$, $\theta\mapsto\Gamma(\theta,{\sigma})$ is a smooth convex (curvature$\geq 0$) embedded curve (denoted by $\Gamma({\sigma})$).

Assume that $0<a_0 < a(\theta,{\sigma})<a_1 <+\infty$, and also assume that \textbf{the origin is enclosed by $\Gamma({\sigma})$} for ${\sigma}\in [0,+\infty)$. Then there are only two possible cases:

(1) There exists $A>0$ such that the area enclosed by $\Gamma({\sigma})$ is greater than $Ae^{2{\sigma}}$;

(2) There exists $0<r<R$ such that $\Gamma(\theta,{\sigma})\in B_R\setminus B_r$ for each $(\theta,\sigma)\in S^1\times[0,\infty)$.

\end{theorem}

Secondly in Section 5-8, we apply this dichotomy to classify the curve shortening flow singularities of closed spacelike curves in $\mathbb{R}^{n+2,k}$ that admit a one-to-one convex projection onto a spacelike plane.

\begin{theorem}\label{main theorem 2}
Let $\Gamma:S^1\times[0,T)\to\mathbb R^{n+2,k}$ be a maximal smooth spacelike CSF. Suppose $P_{xy}\circ\Gamma(\cdot,0)$ is a one-to-one parametrization of a convex planar curve. Then $T<\infty$, and exactly one of the following alternatives holds:

(1) $\displaystyle\limsup_{t\uparrow T}\max_{\theta\in S^1}|\Gamma_s(\theta,t)|_E=\infty$;

(2) There is a point $a$ to which $\Gamma(t)$ shrinks, and the normalized curve
\[
\gamma(\theta,\sigma)=[2(T-t)]^{-1/2}(\Gamma(\theta,t)-a),
\qquad \sigma=-\frac12\log\frac{T-t}{T},
\]
converges smoothly to a centered unit circle $C_\infty$ in a fixed spacelike plane. More precisely, there are smooth reparametrizations $\varphi_\sigma:S^1\to S^1$ such that, for every integer $m\geq0$,
\begin{equation}\label{main-exponential-rate}
\|\gamma(\varphi_\sigma(\cdot),\sigma)-C_\infty\|_{C^m(S^1;E)}
\leq C_m e^{-2\sigma}
\end{equation}
for all sufficiently large $\sigma$. The norms use a fixed auxiliary Euclidean metric, and the constants may depend on the initial curve and $m$.
\end{theorem}

Finally in Section 9, we derive a monotonicity formula for an area-bivector in $\mathbb{R}^{2,1}$, yielding an explicit criterion that rules out point collapse. Using this criterion, we construct a smooth spacelike initial curve with a one-to-one strictly convex projection whose flow develops a null-degenerate singularity. We also discuss the singularity behavior of closed strong spacelike curves, proving round-point convergence in the index-one case and exhibiting null degeneration in an index-two example.\\

\noindent\textbf{Methodology.}
Our setting presents two difficulties in relation to existing approaches. In the curve case considered by Andrews and Zhou \cite{AZ26}, spacelike convexity ensures that the curvature vector is everywhere strictly spacelike. Our convex-projection hypothesis does not impose this condition: the curvature vector may be timelike or null, and the usual positive curvature magnitude need not be globally available. Consequently, the spacelike-curvature structure underlying their pinching and noncollapsing estimates cannot be assumed here. Furthermore, the ambient metric is indefinite, so Huisken’s classical Gaussian monotonicity formula, used in Sun’s Euclidean analysis \cite{Sun26}, is not directly applicable with its usual nonnegative dissipation term. These difficulties motivate an approach based on global geometric control through a planar projection.

Under the convex-projection hypothesis, the projected curves evolve by a weighted planar curve shortening flow (up to tangential reparametrization). This reduces the key global estimates to an analysis of the overall behavior of the planar weighted flow, while retaining the influence of the transverse directions through the weight. In the absence of null degeneration, the inclination estimates and the uniform spacelike bound provide positive upper and lower bounds for this weight. We establish a dichotomy for the corresponding weighted normalized flow using only these bounds, without requiring uniform control of derivatives of the weight. Applied to the projected flow, this yields uniform annular bounds, which in turn provide position bounds for the full normalized curves.

Together with the inclination estimates, this geometric control provides local graphical representations on a uniform spatial scale, with bounded slopes and uniformly parabolic evolution equations. Interior estimates then yield bounds for all higher derivatives and smooth subsequential compactness after suitable reparametrization. To identify the limits, we combine a positive lower bound for the projected curvature with the dissipation of the total Euclidean curvature of the coordinate copy, establishing planarity of the limiting curves. An isoperimetric argument and a subsequent analysis of the normalized equations yield convergence to a round circle and the sharp exponential rate. The resulting framework applies to both Euclidean and pseudo-Euclidean spaces without requiring the original curvature vector to be spacelike.
\vspace{1cm}

\noindent\textbf{Table of notations.}
\begin{center}
\begin{tabular}{ll}
$S^1$ & The parameter circle $\mathbb R/(2\pi\mathbb Z)$.\\
$\Gamma,\ t,\ T$ & Unnormalized CSF, its time, and its maximal time.\\
$\gamma,\ \sigma$ & Normalized CSF and normalized time.\\
$s,\ \hat s,\ \ell$ & Pseudo-Euclidean, planar projected, and Euclidean arclength.\\
$Q$ & Euclidean coordinate copy. \\
$P_\Sigma$ & Pseudo-Euclidean spacelike planar orthogonal projection. \\
$P^*_\Sigma$ & Euclidean planar orthogonal projection.\\
$D(t),\ A_{\mathrm{un}}(t)$ & Domain enclosed by $P_{xy}(\Gamma(t))$ and its area (Section~6).\\
$A(\sigma)$ & Area of the domain enclosed by $P_{xy}(\gamma(\sigma))$ in Sections~6--8.\\
$|\cdot|_E$ & Euclidean norm in fixed ambient coordinates.\\
$\hat k,\ \kappa_E$ & Curvature of $P_{xy}(\gamma)$ and curvature of $Q(\gamma)$.\\
$u,\ f,\ M,\ F$ & Log curvature, its extrema and integral (Theorem~\ref{curvature lower bound}).\\
$\hat g,\ L=\hat L$ & Projected metric density and length (Theorem~\ref{curvature lower bound}).\\
$\mathfrak d=\kappa_E^2\tau_E^2$ & Curvature--torsion density (Theorem~\ref{planar subsequential limit}).\\
$K_E,\ E,\ \mathcal B_E$ & Total curvature, torsion energy and torsion vector (Theorem~\ref{planar subsequential limit}).\\
$\mathscr T,\ \mathscr B,\ \Delta$ & Chord/osculating sets, three-point planes, inclination.\\
$\mathcal Q,\ \mathcal E,\ \mathcal D$ & Graph energies and radial dissipation (Appendix~\ref{sec:exponential-convergence}).\\
$\mathcal A$ & Area bivector in Section~\ref{sec:examples}.\\
$\mathcal F_G$ & Dimensionless Gage deficit used in Section 8.\\
\end{tabular}
\end{center}
\vspace{0.4cm}

\noindent\textbf{Acknowledgement.}
The authors thank helpful discussions with Dong Zhang and Nan Ye which clarified many details in the proofs and improved the exposition in many aspects. We are grateful to Qi Sun, who informed us the area-bivector monotonicity and null-tangent-limit example under the help of ChatGPT. All the figures were produced using GeoGebra.\\

\noindent\textbf{Statement on AI tools.}
The area-bivector monotonicity in Section~9.3 was discovered by ChatGPT. The ideas underlying Example~\ref{explicit-bivector-example} and Appendix ~\ref{sec:exponential-convergence} originated from the authors, while most of the detailed arguments in these parts were developed with the assistance of ChatGPT. In the remaining parts of the paper, ChatGPT only assisted with checking details, writing, and typesetting.

\newpage

\section{The correspondence between CSF and NF}

In the Gage-Hamilton-Grayson Theorem, a solution $\Gamma$ to the embedded curve shortening flow could be applied both temporal and spatial rescalings, and the resulting solution $\gamma$ satisfies the normalized flow equation and converges to the unit circle.

The rescaling process could be viewed as solving the NF with the initial curve
$$
\gamma(0) = \frac{1}{\sqrt{2T_\Gamma}}\Gamma(0)
$$
where $T_\Gamma$ denotes the maximal existence time of $\Gamma$. In this process the size of $\gamma(0)$ is restricted (the area of the interior of $\gamma(0)$ must be $\pi$). More generally, we can remove this restriction on size and replace the scaling factor (of the initial curve) $\frac{1}{\sqrt{2T_\Gamma}}$ by $\frac{1}{\sqrt{2T_0}}$, where $T_0>0$ is arbitrarily chosen.

\begin{proposition}
Let $\Gamma: S^1\times [0,T) \to \mathbb{R}^{p,q}$ be a solution to the CSF with maximal existence time $T$. Suppose $\gamma$ is a solution to the NF, with initial curve $\gamma(0) = \lambda \Gamma(0)$. Then
\begin{align}\label{expression of NF}
    \gamma(u,{\sigma}) = e^{{\sigma}}\sqrt{\frac{1}{2T_0}}\Gamma(u,T_0(1-e^{-2{\sigma}}))
\end{align}
where $T_0 = \frac{1}{2\lambda^2}$. Moreover:\\
\indent If $T_0>T$, then $\gamma$ loses smooth spacelike regularity at ${\sigma}_0 = -\frac{\ln(T_0- T)}{2}+\frac{\ln T_0}{2}$; \\
\indent If $T_0<T$, then $\gamma$ cannot remain in a fixed bounded set as $\sigma\to\infty$.
\end{proposition}

\begin{proof}
If we let
\begin{align}\label{NF TO CSF}
    \Gamma_1(u,t) = e^{-{\sigma}}\sqrt{2T_0}\gamma(u,{\sigma}).
\end{align}
where $t = T_0(1-e^{-2{\sigma}})$, then simple calculation shows that $\Gamma_1$ is exactly the solution to the CSF with initial data $\Gamma_1(0) = \Gamma(0)$. From the uniqueness result of the curve shortening flow, $\Gamma_1 = \Gamma$ and \ref{expression of NF} follows.

For $T_0>T$, the rescaling and time change are nonsingular at the stated finite $\sigma_0$, so an extension of $\gamma$ would extend $\Gamma$ past $T$. For $T_0<T$, choose $u$ such that $\Gamma(u,T_0)\ne0$, which is possible because $\Gamma(\cdot,T_0)$ is an immersion. Formula~\eqref{expression of NF} then gives $|\gamma(u,\sigma)|_E\to\infty$.
\end{proof}

In some points of view, \textbf{the NF solution with $\lambda$-times initial data is an observation of the CSF solution during the time interval $\left[0,T_0 =\frac{1}{2\lambda^2}\right)$.} (This observation involves both a continuously expanding perspective and a slow-motion process.)

Below are examples evolving under the NF: lines, circles and rescaled grim-reapers.

\begin{example}\label{example line}
(\textbf{The Moving Lines.})
Let $\gamma$ be a solution to the NF with $\Gamma(x,0) = x\vec{v} + \vec{w}$. Then $\gamma(x,{\sigma}) = e^{\sigma}(x\vec{v} + \vec{w})$.
\end{example}

\begin{example}\label{example circle}
(\textbf{The Expanding or Shrinking Circles.}) Let $\Gamma$ be a solution to the CSF with initial data
$$
\Gamma(\theta,0) = (\cos \theta, \sin \theta).
$$
By rotation symmetry, $\Gamma(t)$ is a circle for $t\in [0,T)$.

Let $A(t)$ be the area enclosed by $\Gamma(t)$; then $A'(t)=-\int_{\Gamma(t)}k\,ds=-2\pi$. So $A(t) = \pi-2\pi t$ and we deduce
$$
\Gamma(\theta,t) = \sqrt{1-2t}(\cos \theta, \sin \theta), \,T = 1/2.
$$

Let
\begin{align*}
\gamma(\theta,{\sigma}) = [2(T_0- t)]^{-1/2}\Gamma(\theta,t) &= \sqrt{\frac{1-2t}{2T_0-2t}}(\cos \theta, \sin \theta)
 \\&= \sqrt{\frac{1-2T_0+2T_0e^{-2{\sigma}}}{2T_0e^{-2{\sigma}}}}(\cos \theta, \sin \theta)
\end{align*}

where ${\sigma} = -\frac{\ln(T_0- t)}{2}+\frac{\ln T_0}{2}$.

$\gamma$ is a solution to the NF, with initial data
$$
\gamma(\theta,0) = \sqrt{\frac{1}{2T_0}}(\cos \theta, \sin \theta).
$$

If $T_0< 1/2$, then the initial curve is a circle with radius $\sqrt{\frac{1}{2T_0}}>1$. Under the NF, $\gamma({\sigma})$ becomes infinitely large as ${\sigma} \to +\infty$.

If $T_0= 1/2$, then the initial curve is a circle with radius $1$. Under the NF, $\gamma({\sigma})$ remains fixed.

If $T_0> 1/2$, then the initial curve is a circle with radius $\sqrt{\frac{1}{2T_0}}<1$. Under the NF, $\gamma({\sigma})$ shrinks and finally converges to the origin as ${\sigma} \to -\frac{\ln(T_0- 1/2)}{2}+\frac{\ln T_0}{2}$.

Moreover, we consider solution to
$$
\gamma_{\sigma} =\gamma_{ss} + \gamma
$$
with a generic initial data
$$
\gamma(\theta,0) = \lambda(\Gamma(\theta,0)+v),\,v\neq 0.
$$

The solution can be written as
\begin{align*}
\gamma(\theta,{\sigma}) = [2(T_0- t)]^{-1/2}(\Gamma(\theta,t)+v)
\end{align*}
where $T_0 =\frac{1}{2\lambda^2} $ and ${\sigma} = -\frac{\ln(T_0- t)}{2}+\frac{\ln T_0}{2}$.

So we conclude:

If $\lambda<1$, then $\gamma({\sigma})$ converges to a point $[2\left(\frac{1}{2\lambda^2}- \frac{1}{2}\right)]^{-1/2}v$ as ${\sigma} \to -\frac{\ln\left(\frac{1}{2\lambda^2}- \frac{1}{2}\right)}{2}+\frac{\ln \frac{1}{2\lambda^2}}{2}$.

If $\lambda = 1$, then $\gamma({\sigma})$ moves toward infinity with fixed radius.

If $\lambda > 1$, then $\gamma({\sigma})$ becomes larger and larger while its center moves toward infinity.

Furthermore, the initial size and position of the circle determine whether it will cease to enclose the origin during the evolution. If
\[
\lambda^{-2} = 2T_0>1-|v|^2,
\]
then $(\Gamma(\theta,t)+v)$ does not enclose the origin for $t\in \left(\frac{1-|v|^2}{2}, T_0\right)$, which means $\gamma({\sigma})$ does not enclose the origin as ${\sigma} \to +\infty$. Otherwise, if $\lambda^{-2} \leq 1-|v|^2$, then the origin keeps enclosed by the evolving circle.

By Proposition~\ref{rescaling a flow}, the above conclusions also hold for a more generic solution
$$
\gamma_{\sigma} =a^2\gamma_{ss} + \gamma
$$
where $a$ is a constant and
$$
\gamma(\theta,0) = a\lambda(\Gamma(\theta,0)+v),\,v\neq 0.
$$

\begin{figure}[H]
  \centering

  \begin{subfigure}[b]{0.48\linewidth}
    \centering
    \includegraphics[width=\linewidth,height=0.17\textheight,
  keepaspectratio]{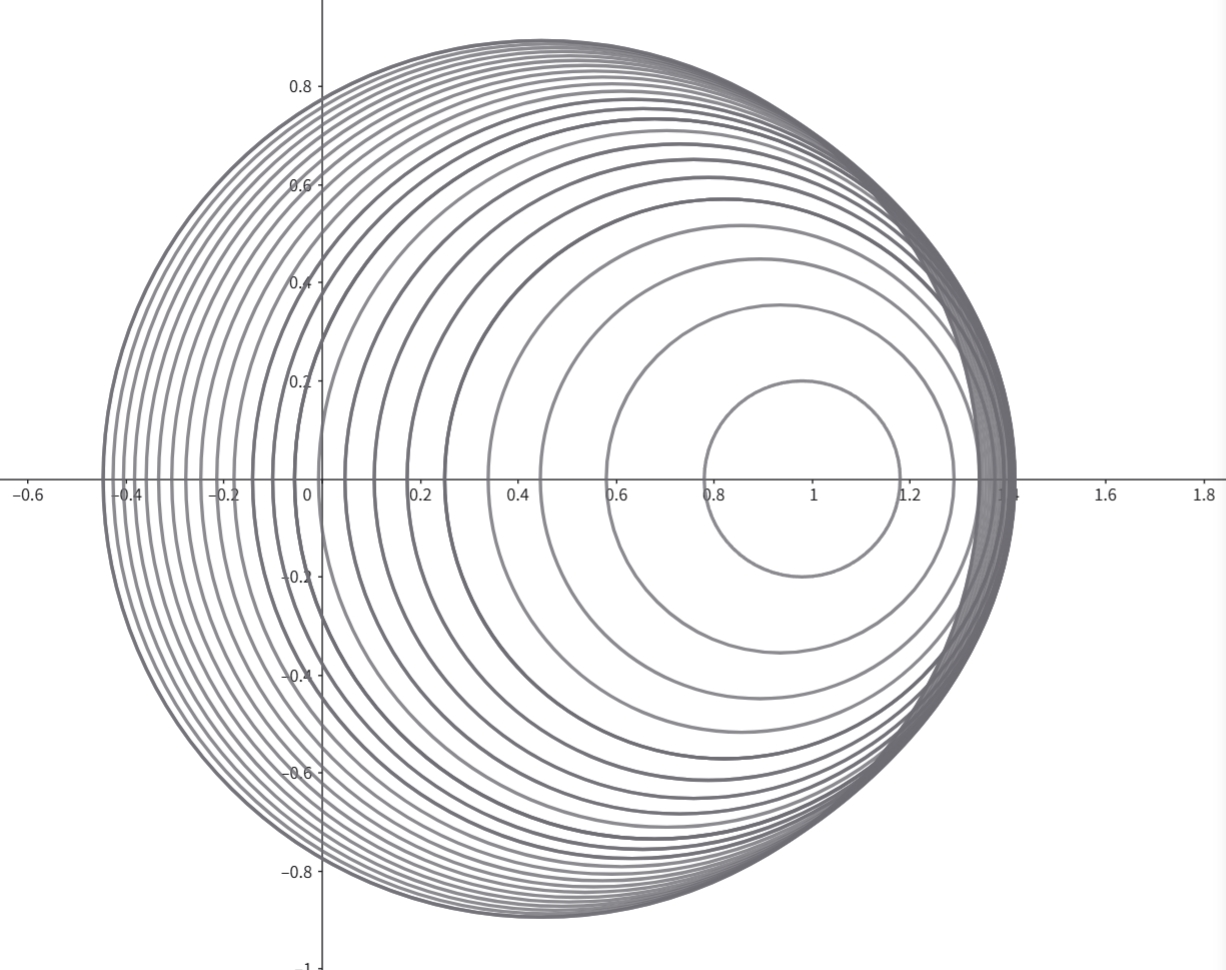}
\caption{The shrinking circles, $|v|=0.5, \lambda = \sqrt{\frac45}$.}
\label{shrinking circle, v = (0,0.5) , lambda = sqrt(frac45)}
  \end{subfigure}
  \hfill
  \begin{subfigure}[b]{0.48\linewidth}
    \centering
    \includegraphics[width=\linewidth,height=0.17\textheight,
  keepaspectratio]{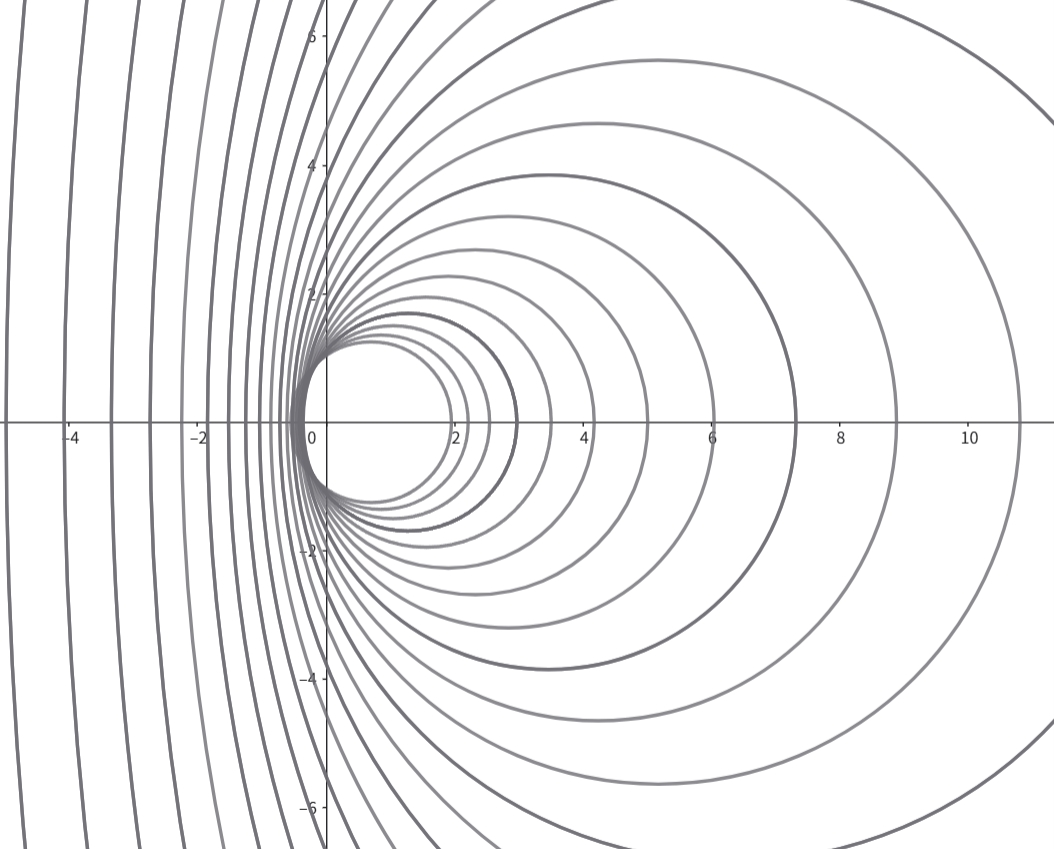}
\caption{The expanding circles, $|v|=0.56, \lambda = 1.25$.}
\label{expanding circle, v = (0,0.56) , lambda = 1.25.jpeg}
  \end{subfigure}

  \par\medskip

  \begin{subfigure}[b]{0.48\linewidth}
    \centering
    \includegraphics[width=\linewidth,height=0.17\textheight,
  keepaspectratio]{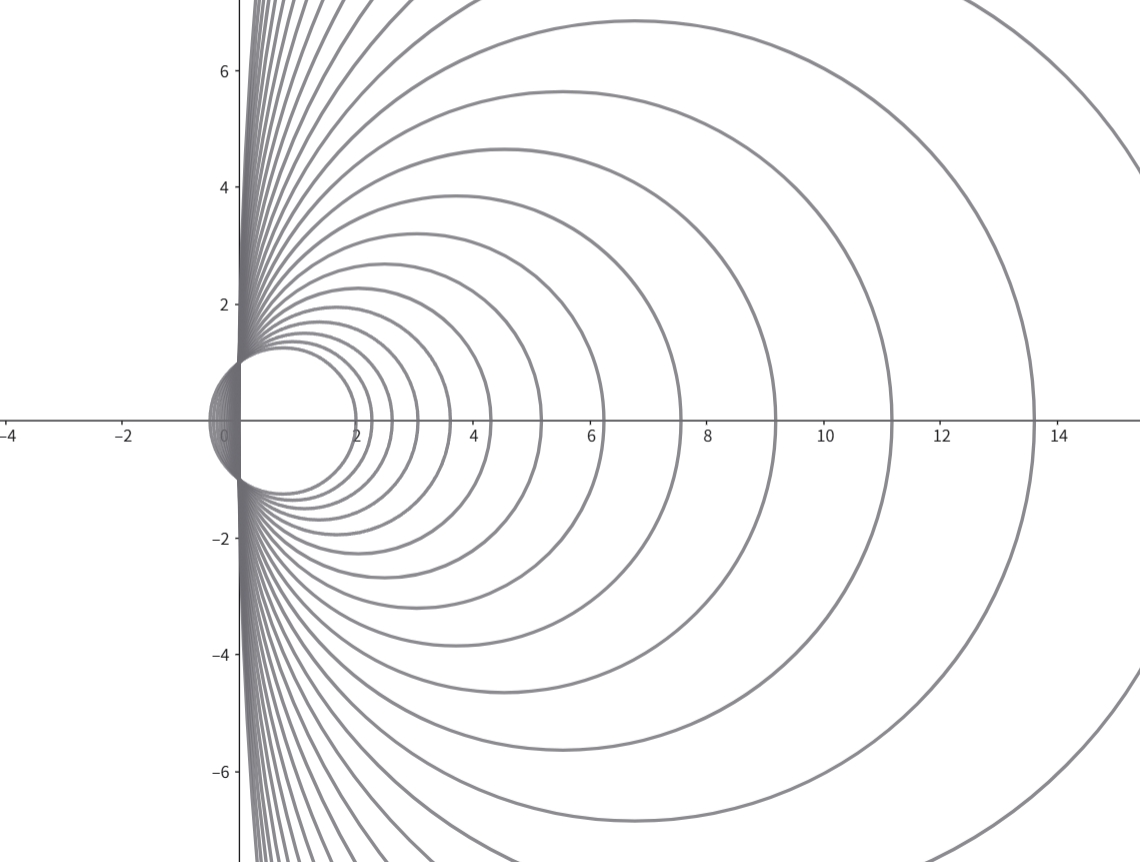}
\caption{The expanding circles, $|v|=0.6, \lambda = 1.25$.}
\label{expanding circle, v = (0,0.6) , lambda = 1.25}
  \end{subfigure}
  \hfill
  \begin{subfigure}[b]{0.48\linewidth}
    \centering
    \includegraphics[width=\linewidth,height=0.17\textheight,
  keepaspectratio]{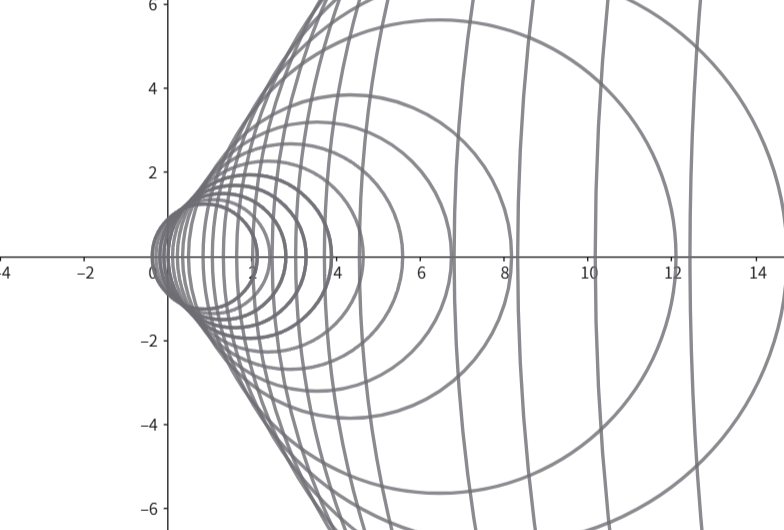}
\caption{The expanding circles, $|v|=0.7, \lambda = 1.25$.}
\label{expanding circle, v = (0,0.7) , lambda = 1.25}
  \end{subfigure}

  \caption{Circles evolving under the normalized flow:
    (a) shrinking circles; (b)--(d) expanding circles.}
  \label{fig:circles}
\end{figure}

\end{example}

\begin{example}\label{example grim reaper}
(\textbf{The Expanding Grim Reapers.}) Let $C$ be any constant and
$$
\Gamma_0(x,t) = (x, -\ln\cos x + t - C),\,x\in (-\pi/2,\pi/2),\,t\in R,
$$
where $C$ is a positive constant. $\Gamma_0$ satisfies
$$
(\Gamma_0)_t = (\Gamma_0)_{ss} + \mu (\Gamma_0)_s
$$
So there exists a reparameterization $x = x(u,t)$, such that
$$
\Gamma(u,t) = \Gamma_0(x(u,t),t)
$$
satisfies
$$
\Gamma_t = \Gamma_{ss}.
$$

Let $T_0> 0$ and
\begin{align*}
\gamma(u,{\sigma}) = [2(T_0- t)]^{-1/2}\Gamma(u,t) = e^{{\sigma}}\sqrt{\frac{1}{2T_0}}\Gamma(u,T_0(1-e^{-2{\sigma}})).
\end{align*}
Then
$$
\gamma_{{\sigma}} = \gamma_{{s}{s}}+\gamma
$$
where ${s}$ is the arc-length parameter of $\gamma$.

The initial data of $\gamma$ is
$$
\gamma(u, 0 ) = \sqrt{\frac{1}{2T_0}}\Gamma(u,0),
$$
which is a rescaling of the grim-reaper by factor $\sqrt{\frac{1}{2T_0}}$.

We shall notice that under the NF, such rescaling of a Grim Reaper expands at the rate $e^{{\sigma}}$.

Now suppose $C>0$. If $T_0> C$, then $\Gamma(u,T_0(1-e^{-2{\sigma}}))$ sweeps through the origin, so $\gamma({\sigma})$ sweeps through the origin.

\begin{figure}[H]
  \centering

  \begin{subfigure}[b]{0.32\linewidth}
    \centering
    \includegraphics[width=\linewidth]{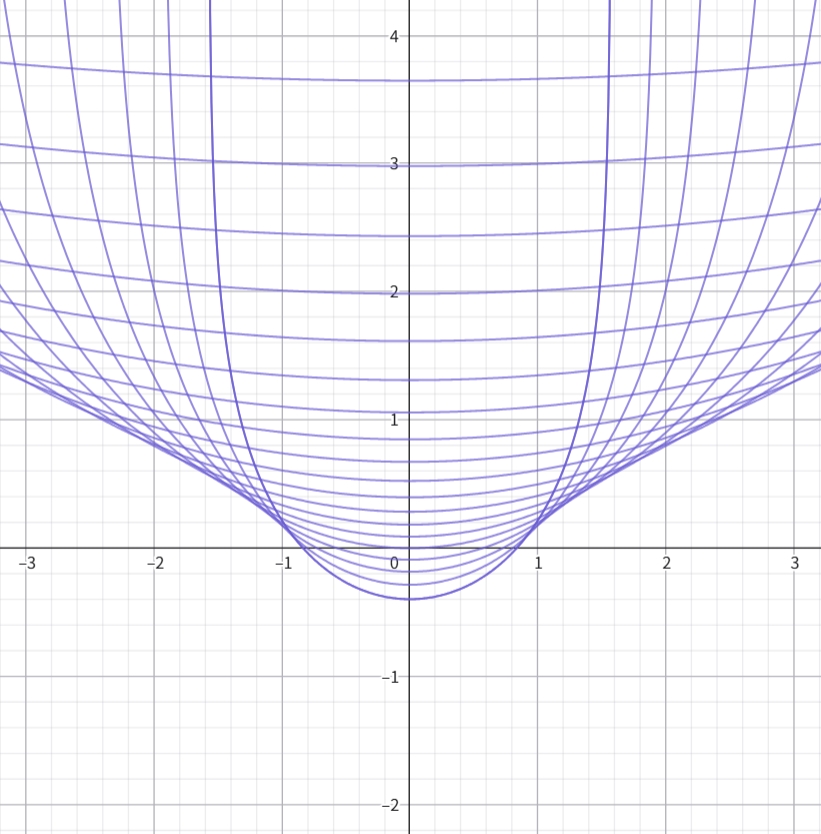}
\caption{The expanding Grim Reapers, $\lambda=1,C=0.4$.}
\label{moving grim reaper, lambda=1,C=0.4}
  \end{subfigure}
  \hfill
  \begin{subfigure}[b]{0.32\linewidth}
    \centering
    \includegraphics[width=\linewidth]{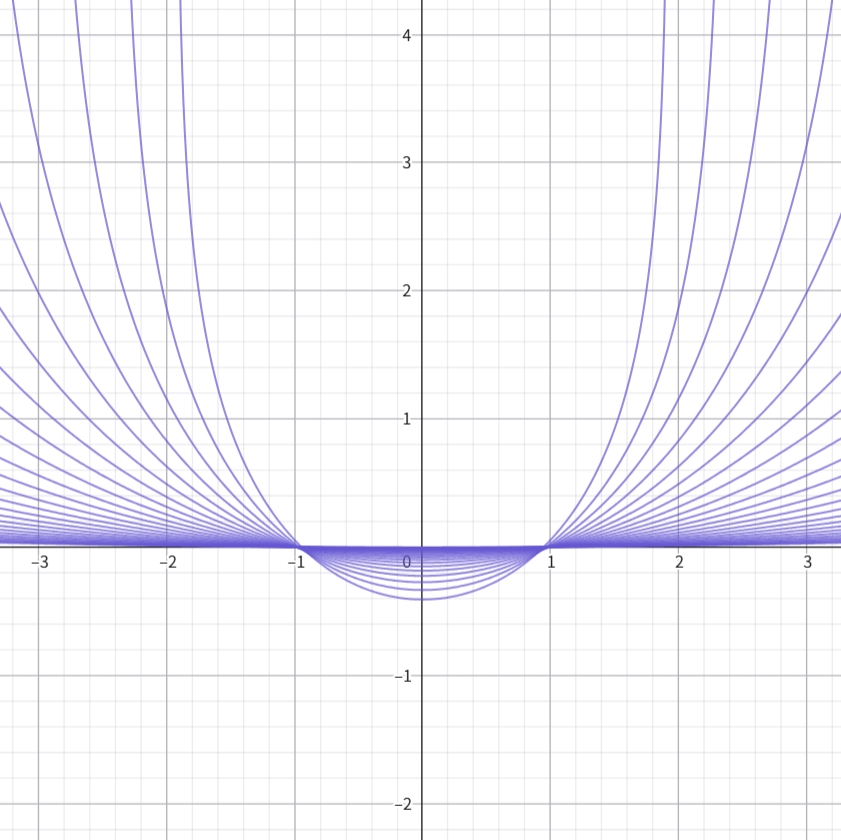}
\caption{The expanding Grim Reapers, $\lambda=1,C=0.5$.}
\label{moving grim reaper, lambda=1,C=0.5}
  \end{subfigure}
  \hfill
  \begin{subfigure}[b]{0.32\linewidth}
    \centering
    \includegraphics[width=\linewidth]{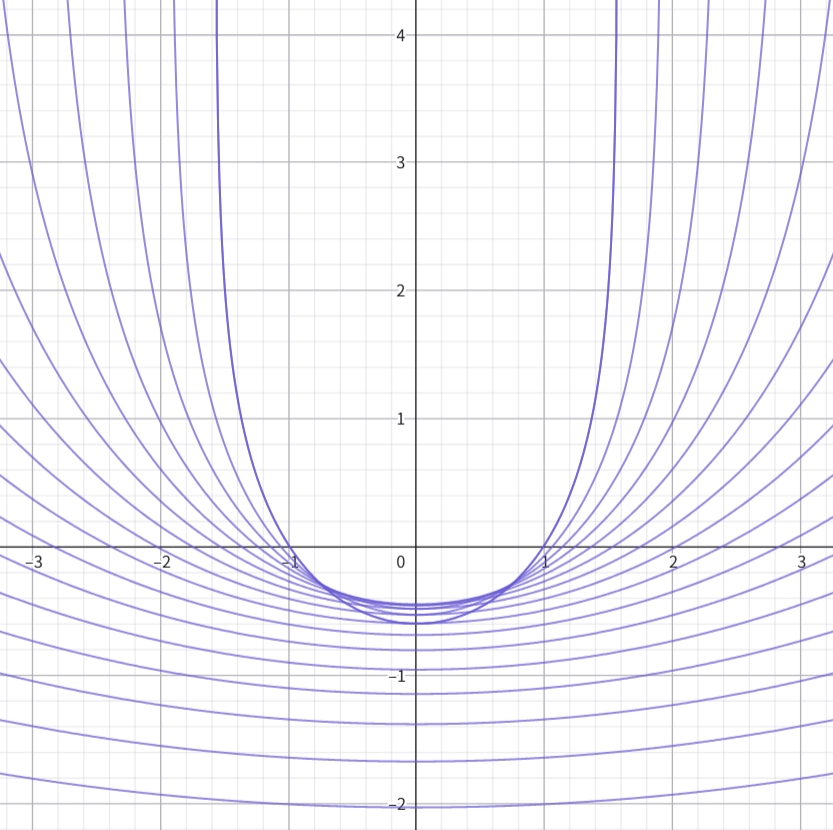}
\caption{The expanding Grim Reapers, $\lambda=1,C=0.6$.}
\label{moving grim reaper, lambda=1,C=0.6}
  \end{subfigure}

  \caption{Rescaled grim reapers evolving under the
    normalized flow with $\lambda=1$.}
  \label{fig:grim-reapers}
\end{figure}

\end{example}

\section{Avoidance and Barriers}
Our proof strongly relies on constructing barriers. So we provide the following avoidance principle for weighted normalized flows.

\begin{lemma}\label{strict barrier lemma}
Let $u(\theta,\sigma),v(\theta,\sigma)$ be two embedded convex planar $C^2$ curves evolving according to the following equations:
\[
u_\sigma=\eta(\theta,\sigma)k(u)N(u)+u,\qquad
v_\sigma=\mu(\theta,\sigma)k(v)N(v)+v,
\]
where $N$ is the inward unit normal and $k\geq0$.
Suppose there exists $C$ such that $0\leq\eta\leq C\leq\mu$, uniformly over the respective parameter domains and times. Moreover suppose $v$ is closed ($u$ can be not closed). For a noncompact outer curve, assume it is proper and bounds a convex region. Denote the corresponding convex closed regions by $K_u(\sigma)$ and $K_v(\sigma)$. If
\[
K_v(0)\subset\operatorname{int}K_u(0),
\]
then
\[
K_v(\sigma)\subset\operatorname{int}K_u(\sigma)
\]
for all common smooth existence times. In particular the two curves remain disjoint. No strict positivity of curvature or strict inequality between the weights is required.
\end{lemma}

\begin{proof}
Suppose that the original curves have a first contact at $\sigma_*>0$. Choose $p_*\in\operatorname{int}K_v(\sigma_*)$ and put
\[
p(\sigma)=e^{\sigma-\sigma_*}p_*.
\]
Since $v$ is closed and depends smoothly on time, there exist $0\leq\sigma_0<\sigma_*$ and $\rho>0$ such that
\[
B_\rho(p(\sigma))\subset K_v(\sigma)
\qquad(\sigma_0\leq\sigma\leq\sigma_*).
\]
Thus $(v-p)\cdot\nu\geq\rho$ for every outward unit normal $\nu$ of $v$.

For $\varepsilon>0$, define the $(1+\varepsilon)$-perturbation
\[
\lambda(\sigma)=1+\varepsilon e^{-(\sigma-\sigma_0)},
\qquad
v_\varepsilon=p+\lambda(v-p)=v+\varepsilon e^{-(\sigma-\sigma_0)}(v-p).
\]
Choose $\varepsilon$ sufficiently small that
$K_{v_\varepsilon}(\sigma_0)\subset\operatorname{int}K_u(\sigma_0)$.
Since $p$ lies in the interior of $K_v$ and $\lambda>1$,
\[
K_v(\sigma)\subset\operatorname{int}K_{v_\varepsilon}(\sigma).
\]
At $\sigma_*$, a contact point of $v$ and $u$ therefore lies in the interior of $K_{v_\varepsilon}$ and on $\partial K_u$. Consequently $K_{v_\varepsilon}(\sigma_*)$ is not contained in $K_u(\sigma_*)$. Hence $v_\varepsilon$ and $u$ have a first contact at some $\tau\in(\sigma_0,\sigma_*]$.

Using $p'=p$ and $k(v_\varepsilon)=\lambda^{-1}k(v)$, we calculate
\[
(v_\varepsilon)_\sigma
=\lambda^2\mu k(v_\varepsilon)N(v_\varepsilon)
+v_\varepsilon+\frac{\lambda'}{\lambda}(v_\varepsilon-p).
\]
At this first contact, the inward normals agree and interior tangency gives
\[
k(v_\varepsilon)\geq k(u)\geq0.
\]
Let $\nu=-N$ be the common outward normal and set
\[
h_\varepsilon=(v_\varepsilon-p)\cdot\nu\geq\lambda\rho>0.
\]
The position terms cancel at contact, so
\begin{align*}
\bigl((v_\varepsilon)_\sigma-u_\sigma\bigr)\cdot N
&=\lambda^2\mu k(v_\varepsilon)-\eta k(u)
-\frac{\lambda'}{\lambda}h_\varepsilon\\
&\geq-\frac{\lambda'}{\lambda}h_\varepsilon>0,
\end{align*}
because $\lambda^2\mu\geq\mu\geq\eta$ and $\lambda'=-(\lambda-1)<0$.
The weight comparison is valid even though the contact parameters on the two curves may differ, because of the uniform bound through $C$. However, first contact from the interior requires the displayed normal relative velocity to be nonpositive, a contradiction. The time dependence of $\lambda$ makes the inequality strict even when both contact curvatures vanish.
\end{proof}

\begin{proposition}\label{rescaling a flow}
Suppose $\Gamma$ is a solution to
$$
\Gamma_{\sigma} = a(\theta,{\sigma})\Gamma_{ss}+\Gamma.
$$
Then $\lambda\Gamma$ satisfies
$$
(\lambda\Gamma)_{\sigma} = \lambda ^2a(\theta,{\sigma})(\lambda\Gamma)_{ss}+\lambda\Gamma.
$$
\end{proposition}

\begin{proof}
The proof is straightforward.
\end{proof}

\begin{lemma}\label{small tube barrier}
Suppose $0<h\leq d,d\cdot h<\frac{\pi}{16}$ and $Q_{d,h} = \{y\geq -d,|x|\leq h\}$. If a family of simple closed convex curve $\gamma:S^1\times[0,+\infty)\to \mathbb{R}^2$ evolves according to
\begin{align*}
    \gamma_{\sigma} = a(\theta,{\sigma})\gamma_{ss}+\gamma
\end{align*}
where $a(\theta,{\sigma})>1$, and $\gamma(0)\subset Q_{d,h}$, then there exists ${\sigma}_0$ such that $\gamma({\sigma})$ does not enclose the origin as ${\sigma} > {\sigma}_0$.
\end{lemma}

\begin{proof}
Let
$$
G_0(x,t)=\left(x,-\ln \cos x+t-\frac{\pi d}{2h}\right).
$$
and by a reparameterization, let
$$
G(u,t) = G_0(x(u,t),t)
$$
which satisfies
$$
G_t = G_{ss}.
$$

We transform the Grim Reaper during the time interval $[0,\frac{\pi^2}{32h^2}]$ into a new solution to the NF which satisfies:
$$
g_{\sigma} = g_{ss}+g,
$$
with the initial data
$$
g(u,0) = \frac{4h}{\pi}G(u,0).
$$

From Example ~\ref{example grim reaper} $\left(\lambda=\frac{4h}{\pi},T_0=\frac{\pi^2}{32h^2},C=\frac{\pi d}{2h}\right)$, and the fact that $\frac{\pi^2}{32h^2}>\frac{\pi d}{2h}$, $g$ sweeps through the origin in finite time.

For $|x|\leq h$, the initial graph has height at most $\frac{4h}{\pi}\log\sqrt2-2d<-d$, since $h\leq d$. Thus $Q_{d,h}$ lies strictly in its convex epigraph. So if $\gamma(0) \subset Q_{d,h}$, then $\gamma(0)$ is enclosed by $\frac{4h}{\pi}G(0)$. By Lemma ~\ref{strict barrier lemma}, $\gamma({\sigma})$ does not enclose the origin after $g$ sweeps through the origin.

\end{proof}

\begin{lemma}\label{barrier for a small ball inside}
For any $R>0$, there exists $r>0$, such that: if a family of simple closed convex curve $\gamma(\theta,{\sigma}) : S^1\times [0,+\infty) \to \mathbb{R}^2$ evolves according to
\begin{align*}
    \gamma_{\sigma} = a(\theta,{\sigma})\gamma_{ss}+\gamma
\end{align*}
where $a(\theta,{\sigma})>1$, $\gamma(0)\subset B_R$, and $\gamma(0)\cap B_r\neq \emptyset$, then there exists ${\sigma}_0$ such that $\gamma({\sigma})$ does not enclose the origin as ${\sigma} > {\sigma}_0$.
\end{lemma}

\begin{proof}
It suffices to consider curves initially enclosing the origin. Choose a smooth closed convex body $K_R\subset\{x\geq0\}$ containing $[0,2R]\times[-R,R]$, with a flat boundary segment on $x=0$ containing $\{0\}\times[-R,R]$. Extend it sufficiently far to the right that $|K_R|>\pi$. Its planar CSF $\Upsilon(t)$ exists beyond $t=1/2$, since its extinction time is $|K_R|/(2\pi)$. It is strictly convex for $t>0$; the strong maximum principle for the $x$ coordinate gives
\[
c_R:=\min_{\Upsilon(1/4)}x>0.
\]
Choose $0<r<\min\{R,c_R\}/2$. If $\gamma(0)$ meets $B_r$, rotate about the origin so that a closest supporting line is $x=-r_0$ with $0<r_0<r$. Convexity and $\gamma(0)\subset B_R$ imply that $\gamma(0)$ lies strictly inside $K_R-(r,0)$.
Normalize $\Upsilon(t)-(r,0)$ using $T_0=1/2$. This gives an outer barrier $\upsilon_\sigma=\upsilon_{ss}+\upsilon$ defined for all $\sigma\geq0$. At $\sigma_*=\frac12\log2$, corresponding to $t=1/4$, the barrier lies in $x>0$. It remains there for later times because the unnormalized convex domains are nested. Lemma~\ref{strict barrier lemma} applies (the barrier is strictly convex at positive times), so $\gamma$ cannot enclose the origin for $\sigma\geq\sigma_*$. If the origin is not initially enclosed, the same conclusion follows from the supporting-line maximum principle.
\end{proof}

\section{Proof of Theorem \ref{main theorem}}

We may assume that $a_0 = 1$ by considering the rescaled flow $\sqrt{\frac{1}{a_0}}\Gamma$. Hence $\Gamma:S^1\times[0,+\infty)\to \mathbb{R}^2$ is a smooth mapping such that
\begin{align}\label{WNF-rescaled}
    \Gamma_{\sigma} = a(\theta,{\sigma})\Gamma_{ss}+\Gamma
\end{align}
where $1  < a(\theta,{\sigma})<a_1 <+\infty$. Also, for each ${\sigma}\in [0,+\infty)$, $\Gamma({\sigma})$ is a simple closed convex curve \textbf{enclosing the origin}.\\

\Needspace{5\baselineskip}
\noindent\textbf{Step 1: Calculating the enclosed area.}

Let $k = k(\theta,{\sigma})$ be the curvature of $\gamma({\sigma})$ at $\gamma(\theta,{\sigma})$. The evolution of the enclosed area $A({\sigma})$ is
$$
A'({\sigma})= 2A - \int_{\Gamma({\sigma})} a(\theta,{\sigma})kds.
$$

Using $1<a(\theta,t)<a_1$, we deduce
\begin{align}\label{area evolution inequality}
    2A - 2\pi a_1 <A' <2A-2\pi.
\end{align}

On the one hand, $\frac{d(A-\pi)}{d{\sigma}}<2(A-\pi)$. Notice that $A$ is positive and defined on $[0,+\infty)$, we conclude $A\geq\pi$.

On the other hand, $(A-\pi a_1)'>2(A-\pi a_1)$. So if there exists $t_0$ such that $A(t_0)> \pi a_1$, then
$$
A({\sigma}) \geq e^{2{\sigma}}\cdot\frac{A({\sigma}_0)-\pi a_1}{e^{2{\sigma}_0}}, \forall {\sigma}\geq {\sigma}_0.
$$

On the earlier compact time interval, decrease the positive exponential-growth constant if necessary. Hence either Theorem~\ref{main theorem}(1) holds for all times, or $\pi\leq A(\sigma)\leq\pi a_1$.\\

\Needspace{5\baselineskip}
\noindent\textbf{Step 2: Bound the curves.}

For convenience let $A_0 = \pi-1,A_1 = \pi a_1+1$, then $A_0<A({\sigma})<A_1$.

Let $R({\sigma}) = \sup_{x\in \Gamma({\sigma})}|x|$. By taking an evolving circle
$$
r(\theta,{\sigma}) = e^{{\sigma}-{\sigma}_0}(R({\sigma}_0)+\epsilon)(\cos \theta ,\sin \theta)
$$
as barriers (which satisfies $r_{\sigma} = 0r_{ss}+ r$) and use the avoidance principle, we deduce that for any ${\sigma}_0,{\sigma}\geq 0$ there holds
\begin{align}\label{a0}
    R({\sigma}_0+{\sigma})\leq e^{{\sigma}}R({\sigma}_0).
\end{align}

Choose $T,R_0$ sufficiently large such that
\begin{align}
\label{a1}&R_0>R(0);\\
\label{a2}&A_1 < e^{T/2}(2+32 A_1)^{-1}\frac{A_0}{256 A_1};\\
\label{a3}&R_0\geq 2\sqrt{2}e^{T}A_1R_0^{-1};\\
\label{a4}&\frac{\pi}{32A_1R_0^{-1}}>2A_1R_0^{-1};\\
\label{a5}&\frac{\pi}{32A_1R_0^{-1}} - 2A_1R_0^{-1}\geq \sqrt{2} \cdot \frac{R_0}{16A_1};\\
\label{a6}&\frac{128\sqrt{2}a_1A_1^2}{R_0^2}\leq (2+32 A_1)^{-1}\frac{A_0}{512 A_1}.
\end{align}

\begin{lemma}\label{the radious cannot keep large}
Let $T_0\in [0,+\infty)$. If $R_0\leq R(T_0)\leq 2R_0$, and $R(\sigma)\geq R_0$ for all ${\sigma}\in [T_0,T_0+T]$, then $A(T_0+T) > A_1$ (which leads to a contradiction).
\end{lemma}

\begin{proof}
Set $R=R(T_0)$ and choose $a_0\in\Gamma(T_0)$ with $|a_0|=R$. Without loss of generality let $a_0 = (R,0)$, then by convexity of $\Gamma(T_0)$ and $A(T_0)<A_1$, we have
$$
\Gamma(T_0)\subset \{(x,y)|-2A_1R_0^{-1}<y<2A_1R_0^{-1}\}.
$$

By comparing to the barrier $r(x,T_0+{\sigma})=(e^{\sigma}x,\pm 2e^{\sigma}A_1R_0^{-1})$ (which are straight lines moving under $r_{\sigma} = 0\cdot r_{ss}+r $), we obtain:
\begin{align}\label{a7}
    \text{for }{\sigma}\in [T_0,T_0+T], \Gamma({\sigma})\text{ remains inside the region }\{ |y| < 2e^{T}A_1R_0^{-1}\}.
\end{align}

Let $a({\sigma})$ be one of the farthest points in $\Gamma({\sigma})$ and let $e_1({\sigma}),e_2({\sigma})$ be unit vectors such that $e_1({\sigma})\parallel a({\sigma}),e_2({\sigma})\perp e_1({\sigma})$. By (\ref{a3}) and (\ref{a7}), we have
\begin{align}\label{a8}
   \text{the angle between  $e_1(\sigma)$ and the $x$-axis is at most $\pi/4$.}
\end{align}

By the area bound $A({\sigma}) < A_1$, we have
\[
|p\cdot e_2|\leq 2A_1R_0^{-1}
\]
for any $p\in \Gamma({\sigma})$. Let $h=2A_1R_0^{-1}$, by (\ref{a4}),
\[
h<\frac{\pi}{16h} = \frac{\pi}{32A_1R_0^{-1}}
\]
Apply Lemma~\ref{small tube barrier} at time $\sigma$ with any longitudinal($e_1$ direction) cutoff $d\in (h,\frac{\pi}{16h})$. Since the origin remains enclosed for all later times, the projection must extend to at least $d$ in both longitudinal directions. Letting $d\to \frac{\pi}{16h} = \frac{\pi}{32A_1R_0^{-1}}$ yields the existence of both $p_+(t)$ and $p_-(t)$ in $\Gamma({\sigma})$, such that
$$
p_+\cdot e_1 \geq \frac{\pi}{32A_1R_0^{-1}},\,p_-\cdot e_1 \leq -\frac{\pi}{32A_1R_0^{-1}}.
$$

By (\ref{a5}) (\ref{a7}) and (\ref{a8}), we may let
$$
p_-\in \left(-\infty,-\frac{R_0}{16A_1}\right]\times (-2e^{T}A_1R_0^{-1} , 2e^{T}A_1R_0^{-1})
$$
and
$$
p_+ \in \left[\frac{R_0}{16A_1},+\infty\right)\times (-2e^{T}A_1R_0^{-1} , 2e^{T}A_1R_0^{-1}).
$$

\begin{figure}[H]
\centering
\includegraphics[scale = 0.2]{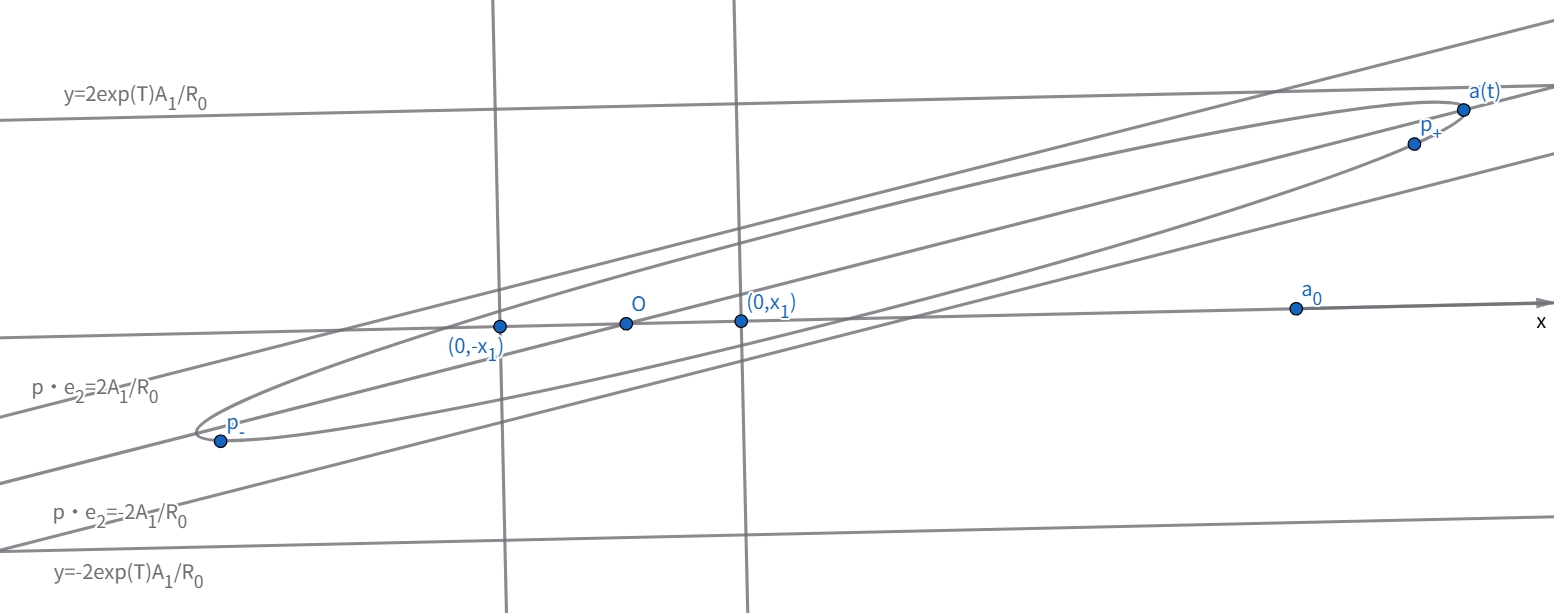}
\caption{Longitudinal extent and transverse control of the convex domain.}
\label{figure curve bounds}
\end{figure}

Let $x_1=R_0/(32A_1)$, $I=(-x_1,x_1)$, and let $F(\sigma)$ be the area of the convex domain inside $|x|<x_1$.

Write $h(x)$ for the vertical chord length at time $T_0$. This nonnegative function is concave on its support, whose length is at most $4R_0$ and which contains $[-2x_1,2x_1]$. Since $\int h>A_0$, there is $x_*$ with $h(x_*)>A_0/(4R_0)$. If $x_*\geq0$, join $(x_*,h(x_*))$ to $(-2x_1,0)$; if $x_*<0$, use $(2x_1,0)$. Concavity gives
\[
h(0)\geq \frac{h(x_*)}{1+32A_1}
 >\frac{A_0}{4R_0(1+32A_1)}.
\]
Concavity on $[-2x_1,2x_1]$ also implies $h(x)\geq h(0)/2$ for $|x|\leq x_1$. In particular,
\begin{equation}\label{a9}
F(T_0)>\frac{A_0}{256A_1(2+32A_1)}.
\end{equation}

Then we calculate the differential of $F({\sigma})$, by viewing the two components of $\Gamma(t)\cap \{|x|<x_1\} $ as graphs over $x$. Let $\{(x,y^+(x))|x\in I\}$ be the component of $\Gamma(t)\cap \{|x|<x_1\} $ which intersects with the positive $y$-axis, and $\{(x,y^-(x))|x\in I\}$ be the other component. Then
$$
y^+_t = y^+-xy^+_x + a\frac{y^+_{xx}}{1+(y^+_x)^2}.
$$
Notice that
$$
\frac{y^+_{xx}}{1+(y^+_x)^2} = -k(x,y^+)\frac{ds}{dx}.
$$
The graph equation also holds for $y^-$, whose curvature term has the opposite sign. So
\begin{align*}
F'({\sigma}) &= \int_{I} (y^+-xy^+_x)dx -\int_{I} (y^--xy^-_x)dx -\int_{\Gamma({\sigma})\cap \{|x|<x_1\}}a kds\\
&\geq 2F-[x(y^+-y^-)]_{-x_1}^{x_1}-a_1\int_{\Gamma({\sigma})\cap\{|x|<x_1\}}k\,ds\\
&= 2F - x_1(y^+(x_1)+y^+(-x_1)-y^-(x_1)-y^-(-x_1))- a_1\int_{\Gamma({\sigma})\cap \{|x|<x_1\}} kds\\
&\geq F-a_1\int_{\Gamma({\sigma})\cap \{|x|<x_1\}} kds.
\end{align*}
The last inequality holds because $x_1(y^+(x_1)+y^+(-x_1)-y^-(x_1)-y^-(-x_1))$ is the area of trapezoid $ABCD$.

Next we estimate $\int_{\Gamma({\sigma})\cap \{|x|<x_1\}} kds$.

\begin{figure}[H]
\centering
\includegraphics[scale = 0.7]{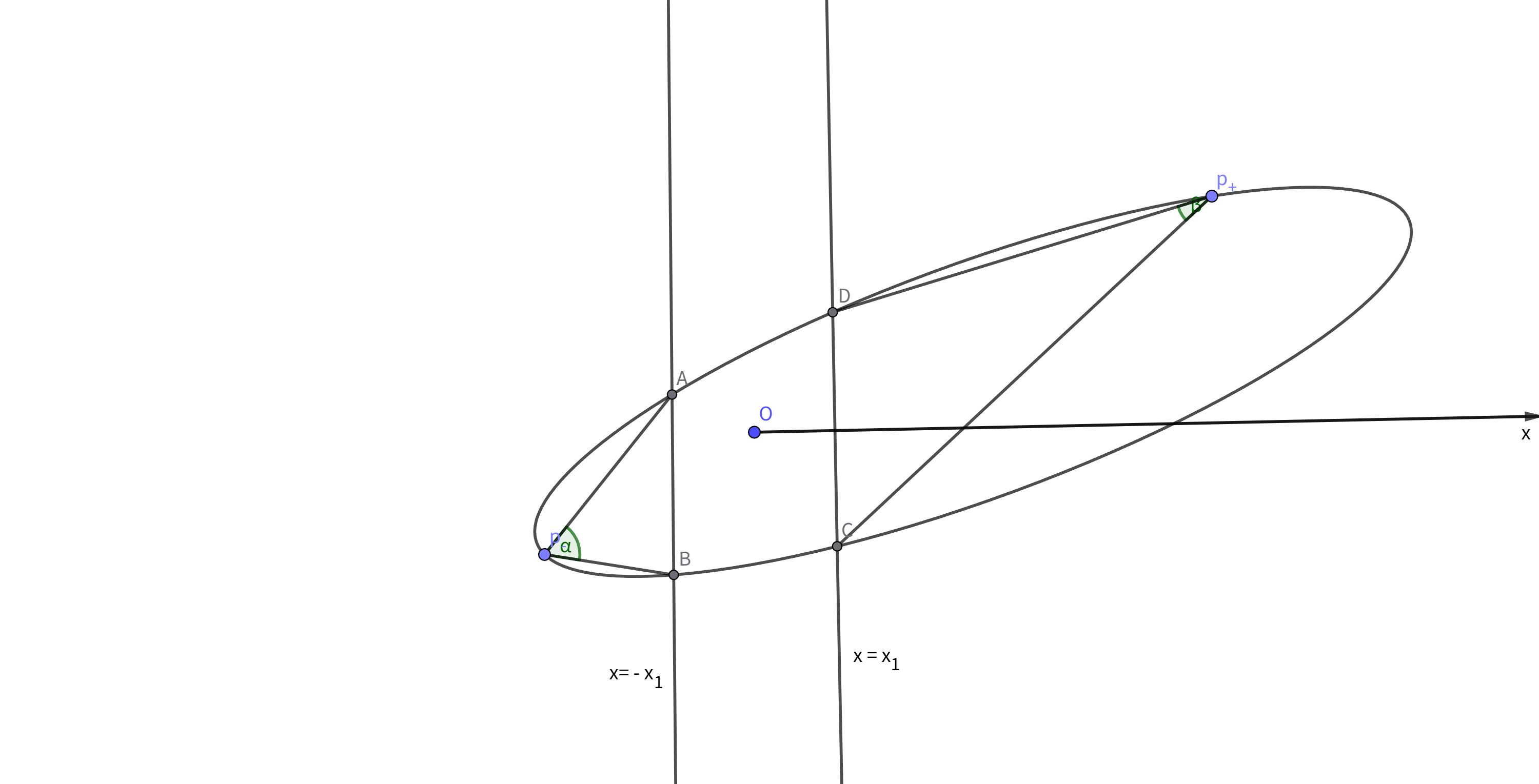}
\caption{The turning angles of the two arcs in the central strip are controlled by exterior chords.}
\label{figure in-slice curvature control}
\end{figure}

In Figure ~\ref{figure in-slice curvature control}, we have
$$
dist(p_-,AB)\geq \frac{R_0}{32 A_1}
$$
and (notice that the width of $\Gamma(t)$ in the $x$ direction is at least $\frac{R_0}{\sqrt{2}}$),
$$
AB\leq \frac{2A_1}{\left(\frac{R_0}{\sqrt{2}}\right)} \leq \frac{2\sqrt{2}A_1}{R_0}.
$$
So
$$
\alpha \leq 2\arctan \frac{AB}{2dist(p_-,AB)} \leq \frac{64\sqrt{2}A_1^2}{R_0^2}.
$$

For $\beta$ we have the same conclusion so

\begin{align*}
\int_{\Gamma({\sigma})\cap \{|x|<x_1\}} kds &\leq \alpha + \beta
\leq \frac{128\sqrt{2}A_1^2}{R_0^2}.
\end{align*}

By \eqref{a6} and \eqref{a9}, $a_1\int_{\Gamma(\sigma)\cap\{|x|<x_1\}}k\,ds\leq F(T_0)/2$. Thus $F'\geq F-F(T_0)/2$. Scalar comparison gives $F\geq F(T_0)$ on the interval, hence $F'\geq F/2$, and we conclude
$$
F(T_0+T)\geq e^{T/2}F(T_0)\geq e^{T/2}(2+32 A_1)^{-1}\frac{A_0}{256 A_1},
$$

By (\ref{a2}), $A(T_0+T)\geq F(T_0+T) > A_1$.
\end{proof}

Now turn back to the proof of the main theorem. Assume that $R({\sigma}_0)>e^TR_0$. Then by \ref{a0} and \ref{a1}, choose $\sigma_1$ to be the last time before $\sigma_0$ at which $R(\sigma_1)=R_0$. Continuity, \eqref{a0}, and \eqref{a1} give $\sigma_1<\sigma_0-T$ and:

(1) $R({\sigma}_1)=R_0$;

(2) During a time interval ${\sigma}\in [{\sigma}_1,{\sigma}_1+T]$, $R({\sigma})\geq R_0$.

This leads to a contradiction with Lemma \ref{the radious cannot keep large}. So we conclude that $\Gamma({\sigma})\subset B_{e^TR_0}$.

Theorem \ref{main theorem} now follows from Lemma \ref{barrier for a small ball inside}.

\section{Application: Singularities of CSF in pseudo-Euclidean spaces}

\noindent\textbf{Curve shortening flow in pseudo-Euclidean spaces.}

Let $\Gamma:S^1\times[0,T)\to \mathbb{R}^{n+2,k}$ be a smooth mapping such that $\theta\mapsto\Gamma(\theta,t)$ is a smooth spacelike immersion for each $t\in [0,T)$, and
$$
\Gamma_t = \Gamma_{ss}
$$
where
$$
\frac{\partial}{\partial s} = \frac{1}{|\Gamma_\theta|}\frac{\partial}{\partial \theta}
$$
denotes the partial derivative with respect to the arc-length parameter. Also, we require $T$ to be the maximal existence time.

In pseudo-Euclidean spaces, a spacelike curve may develop null tangents under the curvature flow. Precisely, null degeneration means $\limsup_{t\uparrow T}\max_{S^1}|\Gamma_s|_E=\infty$; equivalently, a sequence of Euclidean unit tangent directions approaches the null cone. \\

\noindent\textbf{The flow under projection.}

Suppose $\Sigma$ is a spacelike $2$-plane in $\mathbb{R}^{n+2,k}$ and $P_\Sigma$ is the orthogonal projection to $\Sigma$. Let $\hat \Gamma = P_\Sigma\circ \Gamma$ and denote by $\frac{\partial}{\partial\hat s}$ the partial derivative with respect to the arc-length parameter of $\hat \Gamma$.

Then
\begin{align}
\hat \Gamma_t = \hat \Gamma_{ss} = \left(\frac{d\hat s}{ds}\right)^2(\hat \Gamma)_{\hat s\hat s} + \frac{d^2\hat s}{ds^2}(\hat \Gamma)_{\hat s}.
\end{align}

Under a reparameterization (modulo the tangent flow), we may assume
\begin{align}\label{Projection of Gamma}
    \hat \Gamma_t = \left(\frac{d\hat s}{ds}\right)^2(\hat \Gamma)_{\hat s\hat s}.
\end{align}

Moreover, if we consider the normalized flow
\begin{align}
\gamma(\theta, {\sigma}) = [2(T-t)]^{-1/2}\Gamma(\theta,t),
\end{align}
where ${\sigma} = -\frac{\ln(T-t)}{2}+\frac{\ln T}{2}$. Then
$$
\gamma_{\sigma} = \gamma_{ss}+\gamma.
$$

Similarly, there exists a reparameterization $\hat \gamma$ of $P_\Sigma\circ \gamma$,
\begin{align}
    \hat{\gamma}_{\sigma}  = \left(\frac{d\hat s}{ds}\right)^2\hat{\gamma}_{\hat s\hat s} +\hat \gamma.
\end{align}
where $\hat s$ denotes the arc-length parameter of $\hat \gamma$.\\

\noindent\textbf{An alternative theorem for singularities.}

Based on Theorem \ref{main theorem}, we will prove the following singularity classification result.

Sections~\ref{sec:tangent-estimates}--\ref{sec:isoperimetric-convergence} prove the qualitative round-point alternative ,\textbf{under the assumption that the unit tangents $\Gamma_s$ are uniformly bounded}. Appendix~\ref{sec:exponential-convergence} proves the quantitative estimate~\eqref{main-exponential-rate}, completing the proof of Theorem~\ref{main theorem 2}.

\section{Proof of Theorem \ref{main theorem 2}, part \romannumeral1: Euclidean projection and maximal inclination}\label{sec:tangent-estimates}

Let $\Gamma: S^1\times [0,T) \to \mathbb{R}^{n+2,k}$ be a solution to the CSF with maximal existence time $T$, $P_{xy}\big |_{\Gamma(0)}$ is injective and the image $P(\Gamma(0))$ is a convex curve in the $xy$-plane.

Let $\mathbb{R}^{2+n+k}$ be the $(2+n+k)$-dimensional Euclidean space, $Q$ be the coordinate copy from $\mathbb{R}^{n+2,k}$ (with a fixed coordinate system) to $\mathbb{R}^{2+n+k}$  and $\Gamma^* = Q \circ \Gamma$.

\Needspace{5\baselineskip}
\noindent\textbf{Step 1: The Sturmian theorem and the convex projection condition.}\\

For a closed curve $l:S^1\to \mathbb{R}^{n+2+k}$ and a $2$-subspace $\Sigma\subset \mathbb{R}^{n+2+k}$, we say $l$ is projectionally convex to $\Sigma$ if $P_\Sigma^*\circ l$ is a one-to-one mapping from $S^1$ to a convex curve in $\Sigma$, where $P_\Sigma^*$ denotes the Euclidean orthogonal projection to $\Sigma$.

Following the Sturmian theorem in \cite{Ang88} and ideas in \cite{Sun24}, we have the following convexity-keeping theorem.
\begin{theorem}\label{convexity keeping}
Suppose $\Sigma$ is a $2$-subspace in $\mathbb{R}^{n+2+k}$ and $\Gamma^*(t_0)$ is projectionally convex to $\Sigma$. Then $\Gamma^*(t)$ is projectionally convex to $\Sigma$ for $t\in [t_0,T)$. Moreover, if $t\in (t_0,T)$, then:

(1)Every tangent vector of $\Gamma^*(t)$ is not perpendicular to $\Sigma$;

(2)The curvature of the planar curve $P_\Sigma^*\circ \Gamma^*(t)$ is nonzero.
\end{theorem}

\begin{proof}
Let $v$ be any nonzero vector in $\Sigma$. Consider
$$
h(\theta,t) = \left\langle\frac{\partial\Gamma^*(\theta ,t)}{\partial\theta},v\right\rangle.
$$

$\Gamma^*(t_0)$ is projectionally convex to $\Sigma$, so $h(\cdot,t_0)$ changes sign twice.

We compute (here $s$ denotes the parameter with respect to $\Gamma$, not $\Gamma^*$)
\begin{align*}
    h_t = \langle(\Gamma^*)_{t\theta},v\rangle = \langle((\Gamma^*)_{ss})_\theta,v\rangle &= \left\langle\left(\left(\frac{\partial\theta}{\partial s}\right)^2(\Gamma^*)_{\theta\theta} + \frac{\partial^2\theta}{\partial s^2}(\Gamma^*)_{\theta}\right)_\theta,v\right\rangle\\
    &=ah_{\theta\theta}+bh_\theta+ch
\end{align*}
where $a = \left(\frac{\partial\theta}{\partial s}\right)^2$, $a,b,c$ are smooth.

By \cite{Ang88}, for every $t\in(t_0,T)$ the function $h(\cdot,t)$ has exactly two zeros, both simple: a multiple zero would force a strict decrease below two at a later time, which is impossible for the derivative of a nonconstant periodic height function. Since $v$ is arbitrarily chosen, $\Gamma^*(t)$ is projectionally convex to $\Sigma$.

Moreover, assume that (1) does not hold, then there exists $(\Gamma^*)_\theta(\theta_0,t_1)$ perpendicular to $\Sigma$, where $t_1\in (t_0,T)$.

Choose a nonzero $v\in \Sigma$ such that the function $f(\theta) = \langle \Gamma^*(\theta,t_1),v\rangle$ does not attain its maximum or minimum at $\theta_0$. Then $h(\cdot,t_1)$ has at least three zeros: the maximum point of $f$, the minimum point of $f$, and $\theta_0$. This contradicts to the Sturmian theorem.

Now we prove (2): By (1), $P_\Sigma^*\circ \Gamma^*$ is smooth for $t\in (t_0,T)$. Let $\hat \Gamma^*$ be a reparameterization of $P_\Sigma^*\circ \Gamma^*$, which satisfies the evolution equation \ref{Projection of Gamma}. It suffices to prove that $\hat \Gamma^*(t)$ is uniformly convex for $t\in (t_0,T)$.

Let
$$
a = \left(\frac{d\hat s^*}{ds}\right)^2,
$$
where $s, \hat s^*$ are arc-length parameters of $\Gamma, P_\Sigma^*\circ \Gamma^*$. Then the curvature $\hat k$ of $\hat \Gamma^*(t)$ satisfies
$$
\hat k_t = (a\hat k)_{\hat s \hat s} + a\hat k^3.
$$
Notice that $a>0$ for $t\in (t_0,T)$. Then by the strong maximum principle, $\hat k>0$ for $t\in (t_0,T)$.

The details of the proof can be found in \cite{Sun24}, Section 2 and Section 3.
\end{proof}

\Needspace{5\baselineskip}
\noindent\textbf{Step 2: Inclination estimate.}

Using the same method in the proof of Theorem \ref{convexity keeping}, we deduce that
$$
J(t)=\{\Sigma\in\operatorname{Gr}(2,2+n+k):
P_\Sigma^*\circ\Gamma^*(t)\text{ is embedded and strictly convex}\}
$$
monotonically increases in $t$.

For any subset $V\subset \mathbb{R}^{2+n+k}$, we define the maximal inclination
$$
\Delta(V) = \sup_{v\in V,v\neq 0} \angle (v,xy\text{-plane}).
$$

\begin{lemma}\label{inclination angle}
Let $E$ be the $xy$-plane and let $\Sigma$ be a two-dimensional subspace of a Euclidean space of dimension at least three.
\begin{enumerate}
\item If $v\ne0$ and $\Delta(\Sigma)+\angle(v,E)<\pi/2$, then $v\notin\Sigma^\perp$.
\item There is $0\ne v\in\Sigma^\perp$ with $\angle(v,E)=\pi/2-\Delta(\Sigma)$.
\end{enumerate}
\end{lemma}
\begin{proof}
If $\Delta(\Sigma)<\pi/2$, write $\Sigma=\{(x,Ax):x\in E\}$. Then $\|A\|=\tan\Delta(\Sigma)$ and
\[
\Sigma^\perp=\{(-A^*y,y):y\in E^\perp\}.
\]
For a nonzero such vector the angle to $E$ is at least $\arctan(1/\|A\|)=\pi/2-\Delta(\Sigma)$, with equality for a maximizing singular vector (and with the usual interpretation when $A=0$). This proves both assertions in this case. If $\Delta(\Sigma)=\pi/2$, the projection of $\Sigma$ onto $E$ has rank at most one; a nonzero vector in $E$ perpendicular to that image proves (2), while (1) is vacuous.
\end{proof}
The purely planar case requires no inclination estimate and is covered by the classical planar result.

Denote
$$
G(t) = \inf_{\Sigma \notin J(t)} \Delta(\Sigma),
$$
then $G(t)$ is monotone increasing in $t$.

Let
\begin{align}\label{three points, the definition of mathscr T}
    \mathscr{T}(t) = \{\Gamma^*(\theta_1,t)-\Gamma^*(\theta_2,t)\}\cup (\bigcup_{\theta\in S^1} \mathrm{Osc}_{\Gamma^*(t)}(\theta))
\end{align}

contain all chords and all vectors on the osculating subspaces of $\Gamma^*(t)$.

\textbf{Claim 1}:
$$
\frac{\pi}{2}-\Delta(\mathscr{T}(t))\leq G(t) \leq \frac{\pi}{2}.
$$

To show this, let $\Sigma$ be any $2$-subspace such that $\Delta(\Sigma)< \frac{\pi}{2}-\Delta(\mathscr{T}(t))$. It suffices to show that $\Sigma\in J(t)$.

By Lemma \ref{inclination angle}, $\Sigma$ is not perpendicular to any vector in $\mathscr{T}(t)$. Each chord of $\Gamma^*(t)$ is not perpendicular to $\Sigma$, so $P_\Sigma^*\big|_{\Gamma^*(t)}$ is a one-to-one projection. Each osculating subspace of $\Gamma^*(t)$ contains no vector perpendicular to $\Sigma$, so the curvature of  $P_\Sigma^*\big|_{\Gamma^*(t)}$ is nonzero. We conclude that $\Gamma^*(t)$ is projectionally convex to $\Sigma$ and $\Sigma\in J(t)$.

Let
$$
\mathscr{B}(t) = \bigcup_{\theta_1,\theta_2,\theta_3\text{ are distinct}}\mathrm{span}\{\Gamma^*(\theta_1,t)-\Gamma^*(\theta_2,t),\Gamma^*(\theta_1,t)-\Gamma^*(\theta_3,t)\}.
$$
Then $\overline{\mathscr{T}(t)} \subset \overline{\mathscr{B}(t)} $, so
$$
\Delta(\mathscr{T}(t)) \leq \Delta(\mathscr{B}(t)).
$$

\textbf{Claim 2}:
$$
\Delta(\mathscr{B}(t)) \leq \frac{\pi}{2}-G(t).
$$

To show this, let $\Sigma_1$ be a subspace parallel to some plane $\Sigma_2$ which includes three different points of $\Gamma^*(t)$. It suffices to show that
$$
\Delta(\Sigma_1)\leq \frac{\pi}{2}-G(t).
$$

By Lemma \ref{inclination angle}, let $v$ be a nonzero vector perpendicular to $\Sigma_1$ such that $\Delta(\{v\}) = \frac{\pi}{2}-\Delta(\Sigma_1)$. Choose $0\ne w\in v^\perp\cap xy\text{-plane}$ and let $\Sigma = \mathrm{span}\{v,w\}$, then $\Delta(\Sigma) = \Delta(\{v\}) = \frac{\pi}{2}-\Delta(\Sigma_1)$.

$P_\Sigma^*$ maps the three points into a straight line (perpendicular to $v$), which means $\Sigma \notin J(t)$ and $G(t)\leq\Delta(\Sigma) = \frac{\pi}{2}-\Delta(\Sigma_1)$. This finishes the proof of Claim 2.

We conclude that
\begin{align}\label{three points}
    \Delta(\mathscr{B}(t)) \leq \frac{\pi}{2}-G(t) \leq \Delta(\mathscr{T}(t)) \leq \Delta(\mathscr{B}(t)).
\end{align}
which yields the following lemma:
\begin{lemma}\label{three point monotonicity}
$\Delta(\mathscr{B}(t)) = \Delta(\mathscr{T}(t)) = \frac{\pi}{2}-G(t) $ monotonically decreases in $t$.
\end{lemma}

\Needspace{5\baselineskip}
\noindent\textbf{Step 3: Bounding the normalized flow.}\\

Fix $0<t_0<T$. We first obtain an inclination bound without assuming a bound for $|\Gamma_s|_E$.

By Theorem~\ref{convexity keeping}, the projection at $t_0$ is a regular embedded strictly convex curve. Chord directions have a compact closure after adding tangent directions; projection is nonzero on that closure. Projection is also an isomorphism on every osculating plane. Compactness therefore gives
$$
\Delta(\mathscr{T}(t_0)) < \frac{\pi}{2}.
$$

Let $c_1 = \tan \Delta(\mathscr{T}(t_0))< +\infty$. Using Lemma \ref{three point monotonicity}, for each $v\in\mathscr B(t)$, $t\geq t_0$,
\begin{align}\label{inclination}
    z_i^2(v),w_j^2(v)\leq c_1^2(x^2(v)+y^2(v)).
\end{align}

Hence
$$
\frac{ds}{d\hat s} \leq \sqrt{1+nc_1^2}.
$$

Let $D(t)$ be the closed convex domain enclosed by $\hat\Gamma(t)$, and denote its unnormalized projected area by
\[
A_{\mathrm{un}}(t):=\operatorname{area}(D(t)).
\]
The projected-area evolution and the preceding inclination bound give
\[
A_{\mathrm{un}}'(t)
=-\int_{\hat\Gamma(t)}\left(\frac{d\hat s}{ds}\right)^2\hat k\,d\hat s
\leq-\frac{2\pi}{1+nc_1^2}.
\]
Since $A_{\mathrm{un}}>0$, the maximal time $T$ is finite. From now on assume that alternative (1) does not occur. Smoothness on compact time intervals then gives a uniform bound for $|\Gamma_s|_E$ on $[0,T)$, and hence
$$
\frac{d\hat s}{ds} \leq c_2
$$
for some constant $c_2$.

The compact convex domains $D(t)$ are nested.

\textbf{If $A_{\mathrm{un}}(t)$ admits a positive lower bound}, then there exists a small ball inside $\bigcap_{t\in [t_0,T)} D(t)$. Without loss of generality suppose $B_r \subset \bigcap_{t\in [t_0,T)} D(t)$, and $D(0)\subset B_R$. Then inside the slice $\Omega(r/2) = \{|x|<r/2\}$, each component of $\hat\Gamma\cap \Omega(r/2)$ can be parameterized by $x$ and viewed as a graph flow. We write
$$
r(x,t) = \Gamma(\theta(x,t),t) = (x,y(x,t),z_1(x,t)...,z_n(x,t),w_1(x,t),...,w_{k}(x,t)).
$$

Apparently, each coordinate function is uniformly bounded.

In the fixed-$x$ graph gauge, the tangential velocity required to keep the first coordinate equal to $x$ cancels the first-order tangential term in $\Gamma_{ss}$. Thus each dependent coordinate, including $y$, satisfies exactly
$$
y_t = \frac{1}{1+y_x^2+\sum (z_i)_x^2-\sum (w_j)_x^2}y_{xx}.
$$

The line segments joining $(\pm r,0)$ and $(x,y(x)) $ are enclosed by $\hat \Gamma(t)$, so
$$
|y_x|\leq \frac{|y|}{r-|x|}\leq \frac{2R}{r}.
$$
So $y_x$ is uniformly bounded.

By \eqref{inclination}, $|(z_i)_x|,|(w_j)_x|\leq c_1\sqrt{1+y_x^2}$. Moreover,
$$
\frac{1+y_x^2}{1+y_x^2+\sum (z_i)_x^2-\sum (w_j)_x^2} =  \left(\frac{d\hat s}{ds}\right)^2 \in \left[\frac{1}{1+nc_1^2},c_2^2\right]
$$
so the graph coefficient has both a positive lower bound and a finite upper bound, uniform in these charts.

By Theorem~\ref{Nash and Schauder estimate}, $y_{xx},(z_i)_{xx},(w_j)_{xx}$ are uniformly bounded in $\Omega(r/16) = \{|x|< r/16\}$.

The curvature vector $\Gamma_{ss}$ of $\Gamma$ can be expressed as
$$
    \Gamma_{ss} = \frac{\langle\Gamma_x,\Gamma_x\rangle\Gamma_{xx} - \langle\Gamma_x, \Gamma_{xx}\rangle\Gamma_x}{\langle\Gamma_x,\Gamma_x\rangle^2},
$$

so the curvature vector is also uniformly bounded in $\Omega(r/16) = \{|x|< r/16\}$.

Rotate the graph directions within the $xy$-plane and cover the curve by finitely many smaller charts. Theorem~\ref{Nash and Schauder estimate} gives uniform bounds for all spatial derivatives and the time derivatives on every terminal time interval. The graph metrics have a uniform positive lower bound. Equivalently, use the polar angle of the convex projection to put these charts on a fixed circle. The resulting immersions converge smoothly as $t\uparrow T$ to a smooth spacelike immersion. Local existence from this limiting immersion extends the geometric flow past $T$, a contradiction. Thus $A_{\mathrm{un}}(t)\to0$ as $t\uparrow T$.

Choose a point in the nonempty intersection $\bigcap_{t<T}D(t)$ and translate its planar coordinates to the origin. It lies in the interior of $D(t)$ for every $t<T$: at any positive time the strictly positive inward normal speed makes $D(t')\subset\operatorname{int}D(t)$ for $t'>t$.

Let $A(\sigma)$ denote the normalized projected area enclosed by $\hat\gamma(\sigma)$. The area scaling gives
\[
e^{-2\sigma}A(\sigma)
=\frac{A_{\mathrm{un}}(T-Te^{-2\sigma})}{2T}\longrightarrow0
\qquad(\sigma\to\infty).
\]

Now we apply Theorem \ref{main theorem}: Case (1) cannot happen due to $e^{-2{\sigma}}A({\sigma})\to 0$, so there exist $r,R$ such that $\hat \gamma(\theta,{\sigma})\in B_R\setminus B_r$ for each $(\theta,\sigma)$.

\begin{proposition}\label{coordinate bound for csf}
$\Gamma(t)$ shrinks to a point.
\end{proposition}

\begin{proof}
Let
$$
\mathrm{Ran}_{z_i}(t) = \{Z_i(\Gamma(\theta,t))|\theta\in S^1\}
$$
be the $z_i$-range of $\Gamma(t)$ and
$$
l_{z_i}(t) = |\mathrm{Ran}_{z_i}(t)| = \sup_{\theta\in S^1} Z_i(\Gamma(\theta,t)) - \inf_{\theta\in S^1} Z_i(\Gamma(\theta,t))
$$
be the $z_i$-span of $\Gamma(t)$. Similarly define $\mathrm{Ran}_{x}(t)$, $\mathrm{Ran}_{y}(t)$, $\mathrm{Ran}_{w_j}(t)$, $l_{x}(t)$, $l_{y}(t)$, $l_{w_j}(t)$.

The annular bound for $\hat\gamma$ gives
$$
\mathrm{Ran}_x(t),\mathrm{Ran}_y(t)\subset \left[-\sqrt{2(T-t)}R,\sqrt{2(T-t)}R\right].
$$

Using \ref{inclination}, $|(z_i)_{\hat s}|,|(w_j)_{\hat s}|\leq c_1$, so
$$
l_{z_i}(t), l_{w_j}(t) \leq c_1\cdot (\text{the length of $\hat \Gamma(t)$}) \leq  2\pi\sqrt{2(T-t)}Rc_1 \to 0.
$$

Hence each coordinate range (\textbf{which are all monotonically decreasing}) converges to a point, so $\Gamma$ converges to a point.
\end{proof}

Translate the limiting point to the origin.

\begin{proposition}\label{coordinate bound for nf}
Each coordinate function of $\gamma$ is uniformly bounded.
\end{proposition}

\begin{proof}
We only prove for $z_i$.

By $\Gamma(t)\to0$, $0\in \mathrm{Ran}_{z_i}(t) $. For $\Gamma(t)$, we have $l_{z_i}(t)\leq 2\pi\sqrt{2(T-t)}Rc_1$, hence for $\gamma({\sigma})$, $z_i\in [-2\pi Rc_1,2\pi Rc_1]$.
\end{proof}

Similar to the previous process, each component of $\gamma\cap \Omega(r/2)$ can be parameterized by $x$ and viewed as a graph flow. More specifically,
$$
y_{\sigma} = y-xy_x+\frac{1}{1+y_x^2+\sum (z_i)_x^2-\sum (w_j)_x^2}y_{xx};
$$

Using the same method as above, $y_x$, $(z_i)_x$, $(w_j)_x$ are all uniformly bounded.

Theorem \ref{Nash and Schauder estimate} yields the following corollary:

\begin{corollary}\label{coordinate derivative bound}
For every integer $k\geq1$ and every $\sigma_*>0$, all $k$-th spatial derivatives of $y,z_i,w_j$, as well as their time derivatives, are uniformly bounded for $\sigma\geq\sigma_*$ on smaller graph charts. Bounds on a compact earlier interval depend additionally on the smooth initial data.
\end{corollary}

The projected domains contain $B_r$ and lie in $B_R$. We may parametrize their boundaries by polar angle about the origin, and lift this parametrization to $\gamma$. Arzela--Ascoli and a diagonal argument therefore give a smooth immersed subsequential limit for every sequence $\sigma_i\to\infty$. In the following arguments we may start at any fixed positive normalized time, when the projected curvature is strictly positive.

\section{Proof of Theorem \ref{main theorem 2}, part \romannumeral2: Planar limit curve}\label{sec:planar-convergence}

In this section, we will show that the subsequential smooth limit of $\gamma({\sigma})$ must be a planar curve. We focus on the normalized flow, its projection, and its copy to Euclidean spaces. %They share the same time parameter, which will be denoted by $t$ instead of ${\sigma}$ in the following sections.

Let $\gamma(\infty)$ be any smooth limit of a subsequence $\gamma({\sigma}_i)$.\\

\Needspace{5\baselineskip}
\noindent\textbf{Step 4: Prove the curvature of the limit curve never vanishes.}

First we prove that the curvature vector of $\gamma(\infty)$ never vanishes.

Let $\hat \gamma_0 = P_{xy}\circ \gamma$.

To show that the curvature vector of $\gamma(\infty)$ never vanishes, we prove a uniform positive lower bound for the curvature of $\hat\gamma_0$. We use its logarithmic curvature integral, defined precisely below. Both the curvature and the measure in this integral belong to the planar projection, so we first fix its parametrization and compute its metric evolution.

\textbf{We fix a reparameterization} $(\hat \theta,{\sigma})$ such that $\hat \gamma(\hat \theta,{\sigma}) = \hat \gamma_0(\theta(\hat \theta,{\sigma}),{\sigma})$ satisfies
$$
\hat \gamma_{\sigma} = a(\hat\theta,{\sigma})\hat \gamma_{\hat s \hat s} + \hat \gamma
$$
where $\hat s$ denotes the arc-length parameter of $\hat \gamma$ and $a(\hat\theta,{\sigma}) = \left(\frac{\partial \hat s}{\partial s}\right)^2$, evaluated in this reparametrization.

Let $\hat k(\hat \theta,{\sigma})$ be the curvature of $\hat \gamma({\sigma})$ at $\hat \gamma(\hat \theta,{\sigma})$.

We calculate
\begin{align*}
&{\hat k}_{\sigma} = (a{\hat k})_{\hat s\hat s} + a{\hat k}^3 -{\hat k}
\end{align*}

So
\begin{align*}
({\ln\hat k})_{\sigma} &= \hat k^{-1}( (a{\hat k})_{\hat s\hat s} + a{\hat k}^3 -{\hat k}) \\
&=\left(\frac{({a\hat k})_{\hat s}}{\hat k}\right)_{\hat s} + a\left(\frac{\hat k_{\hat s}}{\hat k}\right)^2+ a_{\hat s}\left(\frac{\hat k_{\hat s}}{\hat k}\right)+a\hat k^2 -1\\
&=\left(\frac{({a\hat k})_{\hat s}}{\hat k}\right)_{\hat s} + a({\ln\hat k})_{\hat s}^2+ a_{\hat s}({\ln\hat k})_{\hat s}+a\hat k^2 -1
\end{align*}

Also,
\begin{align*}
({\ln\hat k})_{\sigma} &= \hat k^{-1}( a{\hat k}_{\hat s\hat s} + 2a_{\hat s}{\hat k}_{\hat s} + a_{\hat s\hat s}{\hat k} + a{\hat k}^3 -{\hat k}) \\
&=a\left(\frac{\hat k_{\hat s}}{\hat k}\right)_{\hat s} + a\left(\frac{\hat k_{\hat s}}{\hat k}\right)^2+ 2a_{\hat s}\left(\frac{\hat k_{\hat s}}{\hat k}\right) + a_{\hat s \hat s } + a\hat k^2 -1\\
&=a({\ln\hat k})_{\hat s \hat s} + a({\ln\hat k})_{\hat s}^2+ 2a_{\hat s}({\ln\hat k})_{\hat s} + a_{\hat s \hat s } + a\hat k^2 -1
\end{align*}

\begin{lemma}\label{s and hats derivative bound}
For any $m\in \mathbb N_{>0}$, $\frac{\partial^m s}{\partial \hat s^m}$ is uniformly bounded.
\end{lemma}

\begin{proof}
Fix $t$. Without loss of generality, locally we parameterize $\gamma$ by $x$, write $x = x(\theta)$ and
$$
\gamma(\theta,t) = (x,p_0(x)=y(x),p_1(x),...,p_{n+k}(x)),\,\hat \gamma(\theta,t) = (x,y(x)).
$$

By Corollary \ref{coordinate derivative bound}, each $\frac{\partial^k p_i}{\partial x^k}$ are uniformly bounded.

Then
$$
\frac{\partial s}{\partial x} = \sqrt{1+\sum\alpha_i(p_i)_x^2}, \, \frac{\partial \hat s}{\partial x} = \sqrt{1+y_x^2},
$$
where $\alpha_i = \pm 1$.

$\gamma_s$ is uniformly bounded so $\frac{\partial x}{\partial s}$ is uniformly bounded. Moreover, due to the uniform bounds of $(p_i)_x$, $\frac{\partial s}{\partial x}, \frac{\partial \hat s}{\partial x}$ are also uniformly bounded. We also notice that $\frac{\partial \hat s}{\partial x} \geq 1$.

We have
$$
\frac{\partial^{m+1} s}{\partial \hat s^{m+1}} = \frac{\partial\left(\frac{\partial^m s}{\partial \hat s^m}\right)}{\partial x}\left(\frac{\partial\hat s}{\partial x}\right)^{-1}
$$
and
$$
\frac{\partial\left(\frac{\partial s}{\partial x}\right)}{\partial x} =\big (\sqrt{1+\sum\alpha_i(p_i)_x^2}\,\big)_x = \frac{(1+\sum\alpha_i(p_i)_x^2)_x}{2\frac{\partial s}{\partial x}}
$$
$$
\frac{\partial\left(\frac{\partial \hat s}{\partial x}\right)}{\partial x} =\big (\sqrt{1+y_x^2}\,\big)_x = \frac{(1+y_x^2)_x}{2\frac{\partial \hat s}{\partial x}}
$$

By induction, $\frac{\partial^m s}{\partial \hat s^m}$ can be expanded as a finite sum of
$$
\frac{\text{polynomials in derivatives of $p_i$ with respect to $x$}}{\left(\frac{\partial s}{\partial x}\right)^a\left(\frac{\partial \hat s}{\partial x}\right)^b},
$$
hence $\frac{\partial^m s}{\partial \hat s^m}$ has uniform bounds for each $m$ by Corollary~\ref{coordinate derivative bound}.

\end{proof}

Correspondingly, $\frac{\partial^m \hat s}{\partial s^m}$ is also uniformly bounded for each $m$. Moreover, using the same method as above, we can show that $\left(\left(\frac{\partial \hat s}{\partial  s}\right)^2\right)_{\hat s \hat s}$ and $\left(\left(\frac{\partial \hat s}{\partial  s}\right)^2\right)_{ \hat s}$ are also uniformly bounded.

By Corollary \ref{coordinate derivative bound}, $\hat k = \frac{\sqrt{(1+\left|y_x\right|^2)\left|y_{x x}\right|^2-\left(y_x \cdot y_{x x}\right)^2}}{\left(1+\left|y_x\right|^2\right)^{\frac{3}{2}}}$ is also uniformly bounded.

For this step, write $S^1$ for the fixed parameter circle of $\hat\gamma$ and set
\begin{equation}\label{log-curvature-definitions}
\begin{aligned}
\hat g(\hat\theta,\sigma)&=|\hat\gamma_{\hat\theta}(\hat\theta,\sigma)|_E,
&d\hat s&=\hat g(\hat\theta,\sigma)\,d\hat\theta,\\
L(\sigma)&=\int_{S^1}\hat g(\hat\theta,\sigma)\,d\hat\theta
=\operatorname{Length}(\hat\gamma(\sigma)),\\
u(\hat\theta,\sigma)&=\log\hat k(\hat\theta,\sigma),\\
f(\sigma)&=\min_{\hat\theta\in S^1}u(\hat\theta,\sigma)
=\log\min_{\hat\theta\in S^1}\hat k(\hat\theta,\sigma),\\
M(\sigma)&=\max_{\hat\theta\in S^1}u(\hat\theta,\sigma)
=\log\max_{\hat\theta\in S^1}\hat k(\hat\theta,\sigma),\\
F(\sigma)&=\int_{\hat\gamma(\sigma)}u\,d\hat s
=\int_{S^1}u(\hat\theta,\sigma)\hat g(\hat\theta,\sigma)\,d\hat\theta.
\end{aligned}
\end{equation}
Thus $f,M,F$ are real-valued functions of normalized time, not fixed constants. The extrema exist because $S^1$ is compact and $\hat k>0$ at positive times. The integral defining $F$ uses projected arclength $d\hat s$, not pseudo-Euclidean arclength $ds$. In this step $L=\hat L$, the projected length. This $F$ is unrelated to the Gage deficit $\mathcal F_G$ in Section~8.

\begin{theorem}\label{curvature lower bound}
For each $\sigma_*>0$, the projected curvature $\hat k$ has a uniform positive lower bound on $\sigma\geq\sigma_*$.
\end{theorem}
\begin{proof}
Fix $\sigma_*>0$ and replace $\sigma$ by $\sigma_*+\sigma$, retaining the notation in \eqref{log-curvature-definitions}. In particular, the new initial curve has strictly positive curvature. The preceding estimates give time-independent constants such that
\[
0<L_0\leq L(\sigma)\leq L_1,\qquad
0<\hat k\leq K_1,\qquad a_0\leq a\leq a_1,\qquad
|a_{\hat s}|+|a_{\hat s\hat s}|\leq A_1,
\]
where $a_0>0$. Constants below may depend on these bounds and on this initial time slice, but not on $\sigma$.

\smallskip
\noindent\emph{1. Preliminary bounds and the evolution of $F$.}
Convexity and the turning-number identity give
\[
2\pi=\int_{\hat\gamma(\sigma)}\hat k\,d\hat s
\leq L(\sigma)\max\hat k\leq L_1e^{M(\sigma)}.
\]
Consequently,
\[
\log(2\pi/L_1)\leq M(\sigma)\leq\log K_1,\qquad
L(\sigma)f(\sigma)\leq F(\sigma)\leq L(\sigma)M(\sigma).
\]
Choose $C_0>0$ such that $|M(\sigma)|\leq C_0$ and $F(\sigma)\leq C_0$ for all $\sigma$.

The metric formula in Appendix~B, applied to the planar flow with no additional tangential term, is
\[
\hat g_\sigma=(1-a\hat k^2)\hat g.
\]
For any smooth function $\psi$ on the fixed parameter circle this yields
\[
\frac{d}{d\sigma}\int_{\hat\gamma(\sigma)}\psi\,d\hat s
=\int_{S^1}(\psi_\sigma\hat g+\psi\hat g_\sigma)\,d\hat\theta
=\int_{\hat\gamma(\sigma)}
\bigl(\psi_\sigma+(1-a\hat k^2)\psi\bigr)\,d\hat s.
\]
The logarithmic curvature equation computed above reads
\begin{equation}\label{evolution of ln hat k}
u_\sigma=\left(\frac{(a\hat k)_{\hat s}}{\hat k}\right)_{\hat s}
+a u_{\hat s}^2+a_{\hat s}u_{\hat s}+a\hat k^2-1.
\end{equation}
Substitute $\psi=u$ and integrate the total derivative on the closed curve:
\begin{align*}
F'
&=\int_{\hat\gamma(\sigma)}
\left[
\left(\frac{(a\hat k)_{\hat s}}{\hat k}\right)_{\hat s}
+a u_{\hat s}^2+a_{\hat s}u_{\hat s}
+a\hat k^2-1+(1-a\hat k^2)u
\right]\,d\hat s\\
&=\int_{\hat\gamma(\sigma)}
\left(a u_{\hat s}^2+a_{\hat s}u_{\hat s}
+a\hat k^2-1-a\hat k^2u\right)\,d\hat s+F.
\end{align*}
Young's inequality gives
\[
a u_{\hat s}^2+a_{\hat s}u_{\hat s}
\geq\frac{a}{2}u_{\hat s}^2-\frac{a_{\hat s}^2}{2a}
\geq\frac{a_0}{2}u_{\hat s}^2-\frac{A_1^2}{2a_0}.
\]
Also, $x^2\log x$ is bounded above on $0<x\leq K_1$, since it tends to zero as $x\downarrow0$. Hence $a\hat k^2u\leq C$, although no lower bound for $\hat k$ has yet been proved. Dropping $a\hat k^2\geq0$ and using $L\leq L_1$, we obtain
\begin{equation}\label{entropy-gradient-estimate}
F'\geq\frac{a_0}{2}\int_{\hat\gamma(\sigma)}u_{\hat s}^2\,d\hat s+F-C_1.
\end{equation}

\smallskip
\noindent\emph{2. Oscillation control and a uniform lower bound for $F$.}
The two arcs joining a minimum point of $u$ to a maximum point each contribute at least $M-f$ to its total variation. Thus
\[
\int_{\hat\gamma(\sigma)}|u_{\hat s}|\,d\hat s\geq2(M-f).
\]
Cauchy--Schwarz, followed by $f^2\leq2(M-f)^2+2M^2$, implies
\begin{align*}
\int_{\hat\gamma(\sigma)}u_{\hat s}^2\,d\hat s
&\geq\frac1{L(\sigma)}
\left(\int_{\hat\gamma(\sigma)}|u_{\hat s}|\,d\hat s\right)^2\\
&\geq\frac4{L_1}(M-f)^2
\geq\frac2{L_1}f^2-\frac4{L_1}C_0^2.
\end{align*}
Together with \eqref{entropy-gradient-estimate}, this gives constants $c_1,C_2>0$ such that
\begin{equation}\label{entropy-minimum-estimate}
F'\geq c_1 f^2+F-C_2.
\end{equation}
If $F\leq0$, then $f\leq F/L\leq0$ and therefore
\[
f^2\geq\frac{F^2}{L^2}\geq\frac{F^2}{L_1^2}.
\]
In particular,
\[
F'\geq c_2F^2+F-C_2\quad\hbox{whenever }F\leq0,
\qquad c_2=c_1/L_1^2>0.
\]
Choose $B>0$ so large that
\[
F(0)>-B,\qquad c_2B^2-B-C_2>0.
\]
At a first downward crossing of the level $-B$, differentiability of $F$ would give $F'\leq0$, whereas the preceding inequality gives $F'>0$. Such a crossing is impossible. Enlarging a constant if necessary, we have proved
\begin{equation}\label{Bounds of F}
-C_F\leq F(\sigma)\leq C_F\qquad(\sigma\geq0).
\end{equation}
Returning to the \emph{unconditional} inequality \eqref{entropy-minimum-estimate}, rather than its version restricted to $F\leq0$, we conclude
\[
F'\geq c_1f^2-C_3.
\]
For every $\sigma\geq1$, integration over the preceding unit interval gives
\[
c_1\int_{\sigma-1}^{\sigma}f(\xi)^2\,d\xi
\leq F(\sigma)-F(\sigma-1)+C_3
\leq2C_F+C_3.
\]
Set $C_4=(2C_F+C_3)/c_1$. There is a time
$\sigma'\in[\sigma-1,\sigma]$ such that $f(\sigma')^2\leq C_4$, and hence
$f(\sigma')\geq-\sqrt{C_4}$.

\smallskip
\noindent\emph{3. From the integral estimate to a pointwise curvature bound.}
The nondivergence form of the same equation is
\[
u_\sigma=a u_{\hat s\hat s}+a u_{\hat s}^2
+2a_{\hat s}u_{\hat s}+a_{\hat s\hat s}+a\hat k^2-1.
\]
At a spatial minimum of $u$, one has $u_{\hat s}=0$ and
$u_{\hat s\hat s}\geq0$, so $u_\sigma\geq-A_1-1$ there.
Equivalently, the maximum principle applied to $u+(A_1+1)\sigma$ shows
\[
f(\sigma_2)\geq f(\sigma_1)-(A_1+1)(\sigma_2-\sigma_1)
\qquad(0\leq\sigma_1\leq\sigma_2).
\]

Using the time $\sigma'$ found above,
\[
f(\sigma)\geq f(\sigma')-(A_1+1)(\sigma-\sigma')
\geq-\sqrt{C_4}-A_1-1\qquad(\sigma\geq1).
\]
On $[0,1]$, smoothness and strict positivity give the finite lower bound
$\min_{S^1\times[0,1]}u$. Thus $f\geq-C_*$ for all shifted times, and
\[
\hat k(\hat\theta,\sigma)=e^{u(\hat\theta,\sigma)}
\geq e^{f(\sigma)}\geq e^{-C_*}>0.
\]
Restoring the original time variable proves the theorem.
\end{proof}

\begin{corollary}\label{limit curvature never vanishes}
The curvature vector of every limiting curve is nonzero. The Euclidean curvature of its coordinate copy is strictly positive.
\end{corollary}

\Needspace{5\baselineskip}
\noindent\textbf{Step 5: The limit is planar.}

\begin{theorem}\label{planar subsequential limit}
Every smooth subsequential limit $\gamma(\infty)$ lies in a spacelike two-dimensional linear subspace.
\end{theorem}
\begin{proof}
All estimates below are on a fixed positive-time tail of the normalized flow. Return to the inherited parametrization in which
$\gamma_\sigma=\gamma_{ss}+\gamma$; the curvature integrals and the conclusion are unchanged by this choice.

\smallskip
\noindent\emph{1. The Euclidean-copy flow and its total curvature.}
Let $\gamma^*=Q\gamma$, and denote its Euclidean arclength by $\ell$, keeping $s$ for pseudo-Euclidean arclength. Put $w=d\ell/ds$. The chain rule gives
\[
\partial_s=w\partial_\ell,\qquad
\gamma^*_{ss}=w^2\gamma^*_{\ell\ell}+w_s\gamma^*_\ell.
\]
Consequently,
\begin{equation}\label{euclidean-copy-flow}
\gamma^*_\sigma=b\gamma^*_{\ell\ell}+c\gamma^*_\ell+\gamma^*,
\qquad b=w^2,\qquad c=w_s.
\end{equation}
The term $c\gamma^*_\ell$ is tangential, but must be included in local evolution calculations. The uniform tangent bound implies
\[
1\leq w=|\gamma_s|_E\leq C,\qquad 0<b_0\leq b\leq b_1.
\]
On a graph chart $G(x,\sigma)$, put
$g_M=\sqrt{\langle Q^{-1}G_x,Q^{-1}G_x\rangle}$ and $v=|G_x|_E$.
Thus $g_M$ uses the original pseudo-Euclidean metric, whereas $v$ uses the Euclidean metric. Then
\[
w=\frac{v}{g_M},\qquad
b=\frac{v^2}{g_M^2},\qquad
c=\frac1{g_M}\partial_x\left(\frac{v}{g_M}\right).
\]
The graph estimates and the positive lower bound for $g_M$ therefore control all spatial derivatives of $b,c$. They also give
$d\ell/d\hat s=v/\sqrt{1+y_x^2}\leq C$, whence
$L_E\leq C\hat L\leq L_E^*$.

Let $\kappa_E$ be the Euclidean curvature. The graph calculation in part~2 below gives constants
$0<\kappa_0\leq\kappa_E\leq\kappa_1$.
Define the tangent, principal normal, and higher-codimension torsion vector by
\[
T_E=\gamma^*_\ell,\qquad
N_E=\frac{(T_E)_\ell}{\kappa_E},\qquad
\mathcal B_E=(N_E)_\ell+\kappa_E T_E.
\]
They satisfy
\[
(T_E)_\ell=\kappa_E N_E,\qquad
(N_E)_\ell=-\kappa_E T_E+\mathcal B_E,\qquad
\mathcal B_E\perp T_E,N_E.
\]
We write $\tau_E^2=|\mathcal B_E|_E^2$; no unit binormal is needed when $\tau_E=0$ or in higher codimension.

Define $K_E(\sigma)=\int_{\gamma^*(\sigma)}\kappa_E\,d\ell$.
Applying Proposition~\ref{total curvature evolution} in Appendix~B directly to the Euclidean flow \eqref{euclidean-copy-flow}, with normal weight $b$, tangential coefficient $c$, and $\delta=1$, gives
\[
K_E'(\sigma)=-\int_{\gamma^*(\sigma)}b\kappa_E\tau_E^2\,d\ell.
\]
Set
\begin{equation}\label{euclidean-torsion-energy}
\mathfrak d=\kappa_E^2\tau_E^2,\qquad
E(\sigma)=\int_{\gamma^*(\sigma)}\mathfrak d\,d\ell\geq0.
\end{equation}
Since $b\kappa_E\tau_E^2=(b/\kappa_E)\mathfrak d\geq(b_0/\kappa_1)\mathfrak d$,
\[
K_E'\leq-\frac{b_0}{\kappa_1}E.
\]
For a fixed starting time $\sigma_1>0$ and every $S>\sigma_1$, integration yields
\[
\frac{b_0}{\kappa_1}\int_{\sigma_1}^{S}E(\sigma)\,d\sigma
\leq K_E(\sigma_1)-K_E(S)\leq K_E(\sigma_1).
\]
Letting $S\to\infty$, we obtain $\int_{\sigma_1}^\infty E\,d\sigma<\infty$.

\smallskip
\noindent\emph{2. Graph formulas and the uniform regularity of the dissipation.}
On a graph chart write
\[
G(x,\sigma)=\gamma^*(x,\sigma)=(x,y(x,\sigma),\ldots),\qquad v=|G_x|_E,
\]
and set
\[
\mathcal D=v^2|G_{xx}|_E^2-(G_x\cdot G_{xx})^2.
\]
Here $d\ell=v\,dx$, and the graph estimates give $1\leq v\leq V_1$.
Direct differentiation of $T_E=G_x/v$ gives
\begin{align*}
v_x&=\frac{G_x\cdot G_{xx}}{v},\\
H_E:=\kappa_E N_E
&=\frac1v\partial_x\left(\frac{G_x}{v}\right)
=\frac{v^2G_{xx}-(G_x\cdot G_{xx})G_x}{v^4},\\
\kappa_E&=|H_E|_E=\frac{\sqrt{\mathcal D}}{v^3}.
\end{align*}
Since $G_x=e_1+(G_x-e_1)$ and $G_{xx}\perp e_1$,
\begin{align*}
\mathcal D
&=|G_{xx}|_E^2+|G_x-e_1|_E^2|G_{xx}|_E^2
-\bigl((G_x-e_1)\cdot G_{xx}\bigr)^2\\
&\geq |G_{xx}|_E^2\geq |y_{xx}|^2.
\end{align*}
The projected curvature satisfies
$\hat k=|y_{xx}|/(1+y_x^2)^{3/2}$. By Theorem~\ref{curvature lower bound},
$\hat k\geq\hat k_0>0$, so $\mathcal D\geq\hat k_0^2$ and
$\kappa_E\geq\hat k_0/V_1^3>0$.
An upper bound follows from the upper bounds for $G_x,G_{xx}$.

For completeness, the first differentiated curvature formula is
\begin{align*}
\mathcal D_x
&=2\bigl(v^2G_{xx}\cdot G_{xxx}
-(G_x\cdot G_{xx})(G_x\cdot G_{xxx})\bigr),\\
(\kappa_E)_x
&=\frac{\mathcal D_x}{2v^3\sqrt{\mathcal D}}
-\frac{3\sqrt{\mathcal D}\,v_x}{v^4}.
\end{align*}
Further differentiation produces finite sums with polynomial numerators in the graph derivatives and denominators consisting of powers of $v$ and $\sqrt{\mathcal D}$. Both denominators are bounded away from zero. Corollary~\ref{coordinate derivative bound} therefore bounds all these derivatives, including their time derivatives on smaller charts.

Using $\partial_x=v\partial_\ell$ in the Frenet equations, we obtain
\[
(H_E)_x=(\kappa_E)_xN_E-v\kappa_E^2T_E
+v\kappa_E\mathcal B_E.
\]
The three terms are mutually orthogonal, and hence
\[
|(H_E)_x|_E^2=(\kappa_E)_x^2+v^2\kappa_E^4
+v^2\kappa_E^2|\mathcal B_E|_E^2.
\]
In particular the density in \eqref{euclidean-torsion-energy} has the graph expression
\begin{equation}\label{torsion-density-graph}
\mathfrak d=\frac{|(H_E)_x|_E^2-(\kappa_E)_x^2}{v^2}-\kappa_E^4.
\end{equation}
This formula uses the graph speed $v=|G_x|_E$, not the unit speed
$|\gamma^*_\ell|_E=1$. If $D$ denotes either $\partial_x$ or the time derivative at fixed graph coordinate $x$, differentiating \eqref{torsion-density-graph} gives
\begin{align*}
D\mathfrak d
&=\frac{2}{v^2}
\left((H_E)_x\cdot D(H_E)_x-(\kappa_E)_xD(\kappa_E)_x\right)\\
&\quad-\frac{2Dv}{v^3}
\left(|(H_E)_x|_E^2-(\kappa_E)_x^2\right)
-4\kappa_E^3D\kappa_E.
\end{align*}
Every factor is uniformly bounded by the preceding graph estimates. Thus
\[
|\mathfrak d|+|\mathfrak d_x|+\left|\left.\partial_\sigma \mathfrak d\right|_x\right|\leq C,
\]
where $\left.\partial_\sigma\right|_x$ denotes differentiation in time while holding the graph coordinate $x$ fixed.

To pass to the inherited curve parameter $\theta$, write $x=x(\theta,\sigma)$.
Let $P_x$ denote the linear functional taking the $x$-coordinate of a vector.
The original normalized equation gives
$x_\sigma=P_x(\gamma_{ss})+x$. The graph and coordinate bounds imply
$|x_\sigma|\leq C$. The chain rule then gives
\[
\left.\partial_\sigma \mathfrak d\right|_\theta
=\left.\partial_\sigma \mathfrak d\right|_x+x_\sigma \mathfrak d_x \Rightarrow \left|\left.\partial_\sigma \mathfrak d\right|_\theta\right|\leq C.
\]
Using the moving-curve integral formula \eqref{moving-curve-integral} in Appendix~B, we obtain
\begin{align*}
E'(\sigma)
&=\int_{\gamma^*(\sigma)}
\left(\left.\partial_\sigma \mathfrak d\right|_\theta
+(c_\ell+1-b\kappa_E^2)\mathfrak d\right)\,d\ell,\\
|E'(\sigma)|
&\leq L_E^*
\left(\left\|\left.\partial_\sigma \mathfrak d\right|_\theta\right\|_\infty
+\|c_\ell+1-b\kappa_E^2\|_\infty\|\mathfrak d\|_\infty\right)
\leq C_E.
\end{align*}

\smallskip
\noindent\emph{3. Vanishing torsion and planarity of every subsequential limit.}
$E(\sigma)$ is non-negative and integrable in $[0,+\infty)$ and $E'$ is bounded, so $E(\sigma)\to0$.

Now let $\gamma^*(\sigma_i)$ converge smoothly in the common parametrization to a limit $G_\infty$. Curvature stays positive, so all the expressions above pass to the limit, and
\[
0=\lim_iE(\sigma_i)
=\int_{G_\infty}\kappa_{E,\infty}^2
|\mathcal B_{E,\infty}|_E^2\,d\ell.
\]
The integrand is continuous and nonnegative, and $\kappa_{E,\infty}\geq\kappa_0$.
It follows that $\mathcal B_{E,\infty}=0$ everywhere. Hence the limit curve lies on a plane.

\smallskip
\noindent\emph{4. The limiting plane is spacelike and passes through the origin.}
From the hypothesis that the unit tangents of $\gamma$ are uniformly bounded, the limiting curve $\gamma(\infty)$ must be spacelike. Since the curve is regular and closed, its projective tangent directions cover the entire projective line of the plane containing it. Hence the limiting plane is spacelike.

Suppose finally that $0\notin \text{the limit plane}$. There is a linear function $\lambda$ vanishing at the origin and satisfying $\lambda(\gamma(\theta,\sigma_i))\geq h_0>0$ for some $\sigma_i$ sufficiently large.
The normalized flow equation imply
\[
\partial_\sigma\lambda(\gamma)=\partial_s^2\lambda(\gamma)+\lambda(\gamma).
\]
The maximum principle yields
\[
\lambda(\gamma(\theta,\sigma_i+\sigma))\geq h_0e^\sigma
,
\]
contradicting the uniform coordinate bound of Proposition~\ref{coordinate bound for nf}.
\end{proof}

Fix one subsequential limit. Apply a fixed pseudo-Euclidean isometry taking its spacelike plane to the $xy$-plane. Smooth convergence implies that the new planar projection is embedded and strictly convex at a sufficiently late time. The convexity preservation theorem and all the preceding boundedness estimates therefore apply in these new coordinates.

\section{Proof of Theorem \ref{main theorem 2}, part \romannumeral3: Isoperimetric convergence}\label{sec:isoperimetric-convergence}

\Needspace{5\baselineskip}
\noindent\textbf{Step 6: Prove $\frac{d\hat s}{ds}\to 1$.}

Recall that $\hat \gamma = P_{xy}\circ \gamma$, and under a reparameterization,
$$
\hat{\gamma}_\sigma = \left(\frac{d\hat s}{ds}\right)^2 \hat\gamma_{\hat s\hat s} + \hat \gamma.
$$

By Corollary \ref{limit curvature never vanishes}, $\gamma^*(\infty)=Q\gamma(\infty)$ is uniformly convex.

Recall that $\mathscr{T}(\sigma)$ is the set containing all chords and all osculating $2$-subspaces of $\gamma^*(\sigma)$. For pairs of parameter points separated by a fixed amount, embeddedness gives a positive limiting chord length. Near the diagonal, normalized secants converge uniformly to tangents; this follows by integrating the derivative in the common parametrization. The osculating planes also converge uniformly, since the limiting projected curvature is positive. Thus smooth convergence yields
$$
\Delta(\mathscr{T}(\sigma_i)) \to 0.
$$

By Lemma \ref{three point monotonicity}, we have
$$
 \Delta(\mathscr{T}(\sigma)) \to 0,
$$
so $\frac{d\hat s}{ds}\to 1$.

\Needspace{5\baselineskip}
\noindent\textbf{Step 7: Prove the isoperimetric ratio converges to $4\pi$.}

We use Gage's inequality \cite{Gag83,Gag84}.
\begin{theorem}\label{Gage's isoperimetric monotonicity}
For a closed convex $C^2$ planar curve of length $L$ and area $A>0$,
\[
\int k^2\,ds\geq\frac{\pi L}{A},
\]
with equality if and only if the curve is a circle.
\end{theorem}
Define the dimensionless Gage deficit
\[
\mathcal F_G(\mathcal C)=1-\frac{\pi L(\mathcal C)}
 {A(\mathcal C)\int_{\mathcal C}k^2\,ds}.
\]
It is nonnegative, vanishes precisely on circles, and is continuous under regular $C^2$ convergence.

\begin{lemma}\label{uniform positive Gage functional}
For every $\epsilon>0$ there exist $\mu_0>0$ and $\sigma_0$ such that
\[
I(\sigma):=\frac{\hat L(\sigma)^2}{A(\sigma)}\geq4\pi+\epsilon,
\quad \sigma\geq\sigma_0
\quad\Longrightarrow\quad \mathcal F_G(\hat\gamma(\sigma))\geq\mu_0.
\]
\end{lemma}

\begin{proof}
Suppose, on the contrary, that there exists a time sequence $\sigma_i\to \infty$ such that:
$$
I(\sigma_i)\geq4\pi+\epsilon,\quad \mathcal F_G(\hat\gamma(\sigma_i))\to0.
$$
From Corollary \ref{coordinate derivative bound}, there exists a subsequence of   $\{\hat\gamma(\sigma_i)\}$(reparameterized by the polar angle if necessary) , which converges to some $C^2$ curve $\hat\gamma_0$ with respect to the $C^2$ topology. Hence $F_G(\hat\gamma_0)=0$ and $\hat\gamma_0$ must be a circle, which contradicts to $I(\sigma_i)\geq4\pi+\epsilon$.
\end{proof}

\begin{corollary}\label{converge to a circle}
The isoperimetric ratio $I(\sigma)$ tends to $4\pi$ as $\sigma\to\infty$.
\end{corollary}
\begin{proof}
Write $\eta=(d\hat s/ds)^2$, so $\eta\to1$ uniformly. The projected flow gives
\[
\hat L'=\hat L-\int\eta\hat k^2\,d\hat s,\qquad
A'=2A-\int\eta\hat k\,d\hat s.
\]
Consequently
\[
I'=\frac{\hat L}{A^2}
\left(-2A\int\eta\hat k^2\,d\hat s
+\hat L\int\eta\hat k\,d\hat s\right).
\]
The length and area have positive upper and lower bounds, and $\hat k$ has an upper bound. For any fixed $\epsilon>0$, Lemma~\ref{uniform positive Gage functional} gives, whenever $I\geq4\pi+\epsilon$ at sufficiently large times,
\[
\int\hat k^2\,d\hat s-\frac{\pi\hat L}{A}
\geq\mu_0\int\hat k^2\,d\hat s.
\]
The expression in parentheses in the formula for $I'$ is therefore at most
\[
-2A\mu_0\int\hat k^2\,d\hat s
+\|\eta-1\|_\infty
 \left(2A\int\hat k^2\,d\hat s+2\pi\hat L\right).
\]
The first term is bounded above by a strictly negative constant, whereas the error tends uniformly to zero. Thus $I'\leq-\mu<0$ whenever $I\geq4\pi+\epsilon$ after a sufficiently large time. The ratio must enter the sublevel set $I<4\pi+\epsilon$, since otherwise it would eventually be less than $4\pi$. It cannot cross this level upward because the derivative is negative there. Letting $\epsilon\downarrow0$ proves the claim.
\end{proof}

\Needspace{5\baselineskip}
\noindent\textbf{Step 8: Fix the radius and center.}

Set $J(\sigma)=\int\eta\hat k\,d\hat s$. Then $J(\sigma)\to2\pi$. Since $A$ is bounded and $A'=2A-J$, integration to infinity gives
\[
A(\sigma)=\int_\sigma^\infty e^{-2(u-\sigma)}J(u)\,du\longrightarrow\pi.
\]
Every subsequential limit is therefore a unit circle in the $xy$-plane.

Suppose one such circle is $S^1+v$ with $v\ne0$. The uniform inner ball gives $|v|<1$. Choose
\[
1<\lambda<(1-|v|^2)^{-1/2},\qquad
\lambda^{-1}<\alpha<1.
\]
The circle $\alpha\lambda(S^1+v)$ strictly encloses $S^1+v$, since $|v|<1$ and $\alpha\lambda>1$. For a sufficiently late member $\hat\gamma(\sigma_i)$ of the convergent subsequence, it also strictly encloses $\hat\gamma(\sigma_i)$, and $\eta>\alpha^2$ for all $\sigma\geq\sigma_i$.

The outer barrier $w_\sigma=\alpha^2w_{\hat s\hat s}+w$ starting from this circle sweeps across the origin in finite normalized time. Indeed, after dividing by $\alpha$, Example~\ref{example circle} applies with $\lambda^{-2}>1-|v|^2$. Lemma~\ref{strict barrier lemma} would force $\hat\gamma$ to cease enclosing the origin, a contradiction. Thus $v=0$.

Step 6 shows that every subsequential limit lies in this same $xy$-plane; the preceding radius and center arguments make each one its centered unit circle. Precompactness and uniqueness of all subsequential limits imply smooth convergence of the whole normalized flow. This proves the qualitative round-point alternative of Theorem~\ref{main theorem 2}; its exponential rate is established in Appendix~\ref{sec:exponential-convergence}.

\section{Singularity examples}\label{sec:examples}

\begin{definition}
A finite-time spacelike CSF has a \emph{null-degenerate singularity} if
\[
\limsup_{t\uparrow T}\max_{S^1}|\Gamma_s|_E=\infty.
\]
\end{definition}

Theorem~\ref{main theorem 2} gives a dichotomy, but does not show that both alternatives occur. We first recall the round-point behavior under the strong spacelike hypothesis. We then retain a twofold-projection null-degenerate singularity example based on total curvature, and construct another null-degenerate singularity example satisfying the one-to-one projection hypothesis of the main theorem.

\subsection{Strong spacelike curves of index 1}

\begin{theorem}\label{not becoming null}
Let $\Gamma:S^1\times[0,T)\to\mathbb R^{2,k}$ be a maximal CSF whose initial curve is closed, strong spacelike, and of index 1. Then it shrinks to a point $a$, and $[2(T-t)]^{-1/2}(\Gamma(t)-a)$ converges smoothly to a unit circle in a spacelike plane.
\end{theorem}
\begin{proof}
The $xy$-projection is regular with nonzero curvature; its rotation index is one, so it is embedded and strictly convex. In particular $\Gamma(0)$ is embedded. We next show that every nonzero chord of $\Gamma(0)$ is spacelike. Otherwise let $v$ be a nonspacelike chord and \textbf{use the Euclidean coordinate copy}. There is a two-plane $\Sigma\subset v^{\perp_E}$ with $\Delta(\Sigma)\leq\pi/4$. Write
$$
\Sigma = \{(X,JX)|x\in xy\text{-plane}\}
$$
as a graph over the $xy$-plane. Let
$$
\Sigma_\lambda = \{(X,\lambda JX)|x\in xy\text{-plane}\},\,\,\,\lambda\in [0,1].
$$
Then $\Sigma_0 = xy\text{-plane}$, $\Sigma_1=\Sigma$, and $\Delta(\Sigma_\lambda)\leq \pi/4$.

Every direction in an osculating plane of the strong spacelike curve has inclination strictly less than $\pi/4$. Lemma~\ref{inclination angle} therefore makes each Euclidean orthogonal projection onto $\Sigma_\lambda$ regular with nonzero curvature. By continuity, the projected rotation index stays one. The projected curve to $\Sigma$ is consequently embedded and strictly convex. This contradicts the two endpoints of the chord $v$ having the same projection.

The normalized chord directions have a compact closure obtained by adding tangent directions. Together with the compact family of osculating-plane directions, their strict spacelikeness yields
\[
\Delta(\mathscr T(0))<\pi/4.
\]
Lemma~\ref{three point monotonicity} preserves this strict bound, so $|\Gamma_s|_E$ remains uniformly bounded. The $xy$-projection is embedded and convex by the same index-one argument. Theorem~\ref{main theorem 2} now applies.
\end{proof}
This gives an alternative proof of the corresponding curve case of \cite{AZ26}.

\subsection{The clam-shell example: a twofold convex projection}

The clam-shell construction in \cite[Example 4.1]{YMW16} is
\[
\Gamma_0(\theta)=(\cos\theta,\sin\theta,h(\theta)),
\qquad \theta\in\mathbb R/(4\pi\mathbb Z),
\]
where $h(\theta+2\pi)=-h(\theta)$. In that example
$h'^2+h''^2\leq\varepsilon^2<1$, and the total pseudo-Euclidean curvature satisfies
\[
K(0)\geq\frac{2\pi}{\sqrt{1-\varepsilon^2}}>4\pi
\qquad\text{if }\sqrt3/2<\varepsilon<1.
\]
The displayed piecewise function in that reference is $C^2$. A sufficiently small $4\pi$-periodic smoothing, using an even mollifier, preserves its half-period antisymmetry, the strict inequality $h'^2+h''^2<1$, and $K(0)>4\pi$. In what follows, $\Gamma_0$ denotes such a smooth clam-shell immersion.

\begin{proposition}\label{clam-shell degeneration}
The maximal spacelike CSF starting from $\Gamma_0$ has finite maximal time and a null-degenerate singularity.
\end{proposition}
\begin{proof}
\emph{Preservation of strong spacelikeness.}
As long as the curve is strong spacelike, its pseudo-Euclidean curvature satisfies
\[
k_t=k_{ss}+k^3+k\tau^2.
\]
The maximum principle gives $\min k(\cdot,t)\geq\min k(\cdot,0)>0$. If strong spacelikeness were first lost at a smooth spacelike time $t_*<T$, continuity of $\langle\Gamma_{ss},\Gamma_{ss}\rangle$ would force $k^2$ to reach zero, contradicting this bound. Thus strong spacelikeness persists for all $t<T$. Proposition~\ref{total curvature evolution}, with no normalization, gives
\begin{equation}\label{clam-total-curvature}
K'(t)=\int_{\Gamma(t)}k\tau^2\,ds\geq0,
\qquad K(t)\geq K(0)>4\pi.
\end{equation}

\emph{Descent of the planar projection.}
Let $\mathcal R(x,y,z)=(x,y,-z)$. In the fixed-parameter parabolic gauge
\[
\Gamma_t=\frac{\Gamma_{\theta\theta}}
 {\langle\Gamma_\theta,\Gamma_\theta\rangle},
\]
the maps $\Gamma(\theta+2\pi,t)$ and $\mathcal R\Gamma(\theta,t)$ solve the same initial-value problem. Uniqueness therefore preserves their equality. Writing $\Gamma=(X,z)$ gives
\[
X(\theta+2\pi,t)=X(\theta,t),\qquad
z(\theta+2\pi,t)=-z(\theta,t).
\]
The projected curve descends to the quotient circle $\mathbb R/(2\pi\mathbb Z)$. The induced scalar parabolic equation preserves its convexity and embeddedness, by the same zero-number argument as in Theorem~\ref{convexity keeping}, now applied on the quotient single-covered planar curve (denoted by $\hat \Gamma$).

With $\hat s$ denoting arclength on $\hat \Gamma$, put $p=dz/d\hat s$ along either lifted branch. Then
\[
ds^2=(1-p^2)d\hat s^2,\qquad
\eta=\left(\frac{d\hat s}{ds}\right)^2=\frac1{1-p^2}.
\]
The two branches have opposite $p$ and the same $\eta$, so the projected weighted flow is well defined on $\hat \Gamma$. Its enclosed area $A$ satisfies
\[
A'(t)=-\int_{\hat \Gamma(t)}\eta\hat k\,d\hat s
\leq-2\pi.
\]
Since $A(0)=\pi$, this proves $T\leq1/2$.

\emph{Compactness under the contrary tangent bound.}
Suppose $|\Gamma_s|_E\leq M$ on $[0,T)$. Then
\begin{equation}\label{clam-weight-bound}
|\Gamma_s|_E^2=\frac{1+p^2}{1-p^2},\qquad
1\leq\eta\leq\frac{M^2+1}{2}.
\end{equation}
This replaces the all-chord inclination estimate from Section 6, which cannot be used here: corresponding points on the two sheets determine nonzero vertical chords.

Applying Step 3 in Section \ref{sec:tangent-estimates} on the quotient curve $\hat\Gamma$ and the normalized projection quotient curve $\hat \gamma$ gives uniform inner and outer radii for the normalized projection. Then, locally view the normalized curve $\gamma$ as a graph and applying interior estimates yields all derivative bounds and a smooth subsequential limit of the normalized flow.

\emph{Planarity and the contradiction.}
Applying Step 4 on $\hat\Gamma$ gives a positive lower bound for its normalized curvature.

The Euclidean total-curvature dissipation proof in Step 5 only uses local derivative bounds, positive Euclidean curvature, and a closed parameter domain; it does not require embeddedness. It therefore applies to the full covering curve and shows that every smooth subsequential limit is planar. Its projection is a nondegenerate strictly convex base curve traversed twice, whose total curvature is $4\pi$.

Therefore \eqref{clam-total-curvature} implies
\[
4\pi=\lim_{i\to\infty}K(t_i)\geq K(0)>4\pi,
\]
a contradiction. So the tangent uniform bound is impossible.
\end{proof}

We now give an null-degenerate example with a \emph{one-to-one} convex projection.

\subsection{An area-bivector obstruction to point collapse}
Work first in $\mathbb R^{2,1}=\mathbb R^2\times\mathbb R$. Write
$$
\Gamma=(X,z),\,\,\,\,X\in \mathbb R^2,z\in \mathbb R
$$
and orient the embedded convex projection counterclockwise. Define
\begin{equation}\label{bivector-definition}
\mathcal A(\Gamma)=\frac12\oint\Gamma\wedge d\Gamma
=\mathcal A_1\,e_1\wedge e_2+m_1e_1\wedge e_3+m_2e_2\wedge e_3,
\end{equation}
where
\[
\mathcal A_1=\frac12\oint\det(X,dX)>0,\qquad
m=(m_1,m_2)=\oint X\,dz.
\]
These quantities are translation invariant, and
$\langle\mathcal A,\mathcal A\rangle=\mathcal A_1^2-|m|^2$. In particular, $\mathcal{A}_1$ is the area enclosed by the projected curve $X$.

\begin{lemma}\label{bivector-monotonicity}
Along a smooth spacelike CSF with an embedded strictly convex planar projection, $L_v(t)=\mathcal A_1(t)-v\cdot m(t)$ is strictly decreasing for every fixed unit vector $v\in\mathbb R^2$. Consequently $D(t)=\mathcal A_1(t)-|m(t)|$ is strictly decreasing.
\end{lemma}
\begin{proof}
Parameterize $\Gamma$ by $\theta$, the outer normal angle of the $xy$-projection curve $X$. Let
\[
X_\theta=\rho\mathbf t,\quad z_\theta=\rho p,\quad
\mathbf t=(-\sin\theta,\cos\theta),\quad
\mathbf n=(-\cos\theta,-\sin\theta),\quad q=1-p^2.
\]
Then $\rho>0$ is the radius of curvature of the projection, and $p=dz/d\hat s$ is the slope with respect to projected arclength. Strict spacelikeness gives $q>0$.

We first compute the curvature vector and the tangential correction that keeps $\theta$ fixed. Since $\Gamma_\theta=\rho(\mathbf t,p)$,
\[
\langle\Gamma_\theta,\Gamma_\theta\rangle=\rho^2q,\qquad
ds=\rho\sqrt q\,d\theta,\qquad
\partial_s=\frac1{\rho\sqrt q}\partial_\theta,\qquad
\Gamma_s=\frac{(\mathbf t,p)}{\sqrt q}.
\]
Using $q_\theta=-2pp_\theta$, we obtain
\[
\begin{aligned}
\Gamma_{ss}
&=\frac1{\rho\sqrt q}\partial_\theta
  \left(\frac{(\mathbf t,p)}{\sqrt q}\right)\\
&=\frac{(\mathbf n,p_\theta)}{\rho q}
  -\frac{q_\theta}{2\rho q^2}(\mathbf t,p)\\
&=\frac{(\mathbf n,p_\theta)}{\rho q}
  +\frac{pp_\theta}{\rho q^2}(\mathbf t,p).
\end{aligned}
\]
After a time-dependent reparametrization, the geometric CSF can be written as $\Gamma_t=\Gamma_{ss}+\xi(\mathbf t,p)$, where the added vector is tangent to the curve. Put
\[
V=\frac1{\rho q},\qquad
U=\frac{pp_\theta}{\rho q^2}+\xi.
\]
Then $X_t=U\mathbf t+V\mathbf n$ and $z_t=Vp_\theta+pU$. In the fixed-angle parametrization, $\mathbf t$ is independent of time. Commuting the coordinate derivatives gives
\[
(X_\theta)_t=\rho_t\mathbf t,
\qquad
(X_t)_\theta=(U_\theta-V)\mathbf t+(U+V_\theta)\mathbf n.
\]
Comparison of the normal components yields $U=-V_\theta$, equivalently $\xi=-V_\theta-pp_\theta/(\rho q^2)$. Substitution therefore gives
\begin{equation}\label{bivector-angle-gauge}
X_t=-V_\theta\mathbf t+V\mathbf n,\qquad
z_t=Vp_\theta-pV_\theta,\qquad V=\frac1{\rho q}.
\end{equation}
All time derivatives in the following calculations are taken at fixed $\theta$. The defining line integrals for $\mathcal{A}_1$ and $m$ are unchanged by this orientation-preserving reparametrization.

For the projected area $\mathcal A_1$, differentiate its integral over the fixed interval $[0,2\pi]$:
\[
\mathcal A_1'=\frac12\int_0^{2\pi}
\left(\det(X_t,X_\theta)+\det(X,X_{\theta t})\right)\,d\theta.
\]
Periodicity eliminates the boundary term in integration by parts, so
\[
\begin{aligned}
\int_0^{2\pi}\det(X,X_{\theta t})\,d\theta
&=\left[\det(X,X_t)\right]_0^{2\pi}
  -\int_0^{2\pi}\det(X_\theta,X_t)\,d\theta\\
&=\int_0^{2\pi}\det(X_t,X_\theta)\,d\theta.
\end{aligned}
\]
Consequently,
\[
\begin{aligned}
\mathcal A_1'&=\int_0^{2\pi}\det(X_t,X_\theta)\,d\theta\\
&=\int_0^{2\pi}\det(-V_\theta\mathbf t+V\mathbf n,\rho\mathbf t)\,d\theta
 =-\int_0^{2\pi}\rho V\,d\theta.
\end{aligned}
\]
Since $\rho V=1/q=1/(1-p^2)$, this proves
\begin{equation}\label{bivector-area-evolution}
\mathcal A_1'=-\int_0^{2\pi}\frac{d\theta}{1-p^2}.
\end{equation}

For the mixed component $m=\int_0^{2\pi}Xz_\theta\,d\theta$, differentiation and integration by parts give
\[
\begin{aligned}
m'
&=\int_0^{2\pi}\left(X_tz_\theta+X(z_t)_\theta\right)\,d\theta\\
&=\left[Xz_t\right]_0^{2\pi}
  +\int_0^{2\pi}\left(X_tz_\theta-X_\theta z_t\right)\,d\theta\\
&=\int_0^{2\pi}\left(X_tz_\theta-X_\theta z_t\right)\,d\theta.
\end{aligned}
\]
Substituting~\eqref{bivector-angle-gauge} exhibits the cancellation of the tangential terms:
\[
\begin{aligned}
X_tz_\theta-X_\theta z_t
&=(-V_\theta\mathbf t+V\mathbf n)\rho p
  -\rho\mathbf t(Vp_\theta-pV_\theta)\\
&=-\rho pV_\theta\mathbf t+\rho pV\mathbf n
  -\rho Vp_\theta\mathbf t+\rho pV_\theta\mathbf t\\
&=\rho V\left(p\mathbf n-p_\theta\mathbf t\right).
\end{aligned}
\]
Thus
\[
m'=\int_0^{2\pi}\frac{p}{1-p^2}\mathbf n\,d\theta
   -\int_0^{2\pi}\frac{p_\theta}{1-p^2}\mathbf t\,d\theta.
\]
Because $|p|<1$, the function $\operatorname{arctanh}p$ is smooth and periodic, and
\[
\partial_\theta(\operatorname{arctanh}p)
=\frac{p_\theta}{1-p^2}.
\]
Integration by parts gives
\[
\begin{aligned}
-\int_0^{2\pi}\frac{p_\theta}{1-p^2}\mathbf t\,d\theta
&=-\left[(\operatorname{arctanh}p)\mathbf t\right]_0^{2\pi}
  +\int_0^{2\pi}(\operatorname{arctanh}p)\mathbf t_\theta\,d\theta\\
&=\int_0^{2\pi}(\operatorname{arctanh}p)\mathbf n\,d\theta.
\end{aligned}
\]
Combining the preceding two identities proves
\begin{equation}\label{bivector-mixed-evolution}
m'=\int_0^{2\pi}\left(\frac{p}{1-p^2}+\operatorname{arctanh}p\right)
\mathbf n\,d\theta.
\end{equation}

Rotate the planar coordinates so that $v=e_1$, then
\[
m'\cdot v=\int_0^{2\pi}\left(\frac{p}{1-p^2}+\operatorname{arctanh}p\right)
(-\mathbf \cos \theta)\,d\theta
\]
and
\[
L_v' = \int_0^{2\pi}\left(\left(\frac{p}{1-p^2}+\operatorname{arctanh}p\right)
\mathbf \cos \theta-\frac1{1-p^2}\right)\,d\theta
\]

For $c\in(-1,1)$, the function
\[
\Phi_c(p)=c\left(\frac{p}{1-p^2}+\operatorname{arctanh}p\right)
-\frac1{1-p^2}
\]
satisfies $\Phi_c'(p)=2(c-p)/(1-p^2)^2$. Its unique maximum is attained at $p=c$, where it equals $-1+c\operatorname{arctanh}c$. Letting $c = \cos \theta$ yields
\[
L_v'\leq\int_0^{2\pi}
[-1+\cos\theta\,\operatorname{arctanh}(\cos\theta)]\,d\theta=0.
\]
The logarithmic singularities are integrable. The integral vanishes by integration by parts, using
$(\operatorname{arctanh}(\cos\theta))_\theta=-1/\sin\theta$ on the open half-circles and $\sin\theta\,\operatorname{arctanh}(\cos\theta)\to0$ at their endpoints. Equality would force $p(\theta)=\cos\theta$ away from the endpoints, contrary to the strict spacelike condition at $\theta=0,\pi$. Thus $L_v'<0$.

For $t_2>t_1$, choose $v$ minimizing $L_v(t_1)$. Then
$D(t_2)\leq L_v(t_2)<L_v(t_1)=D(t_1)$. This avoids differentiating $|m|$ at its zeros.
\end{proof}

\begin{theorem}[Bivector obstruction]\label{bivector-obstruction}
Let a smooth closed spacelike initial curve in $\mathbb R^{2,1}$ have a one-to-one strictly convex projection. If
\[
d_0:=\mathcal A_1(0)-|m(0)|\leq 0,
\]
then it cannot shrink to a point and has a null-degenerate singularity. Equivalently,
\[
\liminf_{t\uparrow T}\min_{S^1}(1-p^2)=0.
\]
\end{theorem}
\begin{proof}
Using Theorem \ref{main theorem 2}, it suffices to prove that the curve cannot shrink to a point. Assume that the curve shrinks to a point $p$ as $t\to T$. Then, by the convexity of the projection, $\mathcal{A}_1\to 0$. Moreover,
\[
|m|= \left|\oint (X-p)\,dz\right|
\le \operatorname{diam}(X)\oint |dz|
\le \operatorname{diam}(X)\,\widehat L
\le \pi\,\operatorname{diam}(X)^2.
\]
We conclude
\[
\lim_{t\to T} ( \mathcal{A}_1(t) -|m(t)| ) = 0,
\]
which contradicts $d_0\leq 0$ and Lemma \ref{bivector-monotonicity}.
\end{proof}

\begin{example}[An explicit one-to-one convex-projection example]\label{explicit-bivector-example}
Set
\begin{equation}\label{explicit-slope}
\begin{gathered}
p_0(\theta)=\frac9{10}\frac{\tanh(20\cos\theta)}{\tanh20},\qquad
z_0(\theta)=\int_0^\theta p_0(u)\,du,\\
\Gamma_0(\theta)=(\cos\theta,\sin\theta,z_0(\theta)).
\end{gathered}
\end{equation}
Since $p_0(\theta+\pi)=-p_0(\theta)$, its mean is zero and $z_0$ is smooth and $2\pi$-periodic. The unit-circle projection makes $\Gamma_0$ embedded, and
\[
\langle(\Gamma_0)_\theta,(\Gamma_0)_\theta\rangle
=1-p_0^2\geq1-(9/10)^2=\frac{19}{100}.
\]
In particular the initial data are uniformly spacelike, not merely close to a singular immersion.

Here $\mathcal{A}_1(0)=\pi$ and symmetry gives $m_2(0)=0$, whereas
\[
m_1(0)=\frac9{10}I_{20},\qquad
I_{20}=\frac1{\tanh20}\int_0^{2\pi}
\cos\theta\,\tanh(20\cos\theta)\,d\theta.
\]
Using $\tanh x\geq1-2e^{-2x}$ for $x\geq0$,
\[
I_{20}\geq\frac4{\tanh20}(1-2J),\qquad
J=\int_0^{\pi/2}\cos\theta\,e^{-40\cos\theta}\,d\theta.
\]
Put $u=\pi/2-\theta$ and use $2u/\pi\leq\sin u\leq u$ on $[0,\pi/2]$. Then
\[
J\leq\int_0^\infty u e^{-80u/\pi}\,du=\frac{\pi^2}{6400},
\qquad
I_{20}>4-\frac{\pi^2}{800}>\frac{319}{80}.
\]
Since $\pi<22/7$, this proves
\begin{equation}\label{explicit-defect-bound}
|m(0)|-\mathcal{A}_1(0)>
\frac{2871}{800}-\frac{22}{7}
=\frac{2497}{5600}>0.
\end{equation}
\end{example}

\begin{corollary}\label{both-alternatives-realized}
For every $n\geq0$ and $k\geq1$, both alternatives of Theorem~\ref{main theorem 2} occur in $\mathbb R^{n+2,k}$.
\end{corollary}
\begin{proof}
A round circle in a spacelike two-plane realizes the round-point alternative. Include the curve \eqref{explicit-slope} into $\mathbb R^{n+2,k}$ by setting all additional coordinates to zero. The fixed-parameter parabolic equation and uniqueness keep those coordinates zero, so its flow remains the same $\mathbb R^{2,1}$ example and realizes null degeneration.
\end{proof}
\appendix

\section{Parabolic interior estimates}\label{app:interior-estimates}

\begin{theorem}\label{Nash and Schauder estimate}
Let $U=(y_1,\ldots,y_N,z_1,\ldots,z_K)$ be a smooth solution on
$(-4r,4r)\times(0,S)$, where $S\leq\infty$, of
\begin{equation}\label{graph-system}
U_t=a(U_x)U_{xx}-\delta xU_x+\delta U,\qquad
a(U_x)=\frac1{1+\sum_i(y_i)_x^2-\sum_j(z_j)_x^2},
\end{equation}
with $\delta\in\{0,1\}$. Assume $|U|+|U_x|\leq M$ and
$0<\lambda\leq a(U_x)\leq\Lambda$. For every $\tau>0$ and every integer $m\geq0$, all spatial derivatives $D_x^mU$ and their time derivatives are uniformly bounded on $[-r,r]\times[\tau,S)$. The constants depend on $m,r,\tau,M,\lambda,\Lambda,N,K$, but not on the upper endpoint $S$.
\end{theorem}
\begin{proof}
We keep the differentiated equations explicit, both to track the normalization terms and to explain why the bounds do not depend on the terminal time.

\smallskip
\noindent\emph{1. The equation for the first derivatives.}
Set $V=U_x$. For a vector $V=(V^1,\ldots,V^{N+K})$, write
\[
q(V)=1+\sum_{\alpha=1}^{N+K}\epsilon_\alpha(V^\alpha)^2,
\qquad a(V)=q(V)^{-1},
\]
where $\epsilon_\alpha=1$ for the $y$ components and $\epsilon_\alpha=-1$ for the $z$ components. The assumptions imply
\[
|V|\leq M,\qquad \Lambda^{-1}\leq q(V)\leq\lambda^{-1}.
\]
Thus $a$ and all its derivatives with respect to $V$ are bounded on the range under consideration. In particular,
\begin{align*}
a_{V^\alpha}&=-2\epsilon_\alpha V^\alpha q^{-2},\\
a_{V^\alpha V^\beta}
&=-2\epsilon_\alpha\delta_{\alpha\beta}q^{-2}
+8\epsilon_\alpha\epsilon_\beta V^\alpha V^\beta q^{-3}.
\end{align*}
Here $\delta_{\alpha\beta}$ is the Kronecker symbol, not the normalization parameter $\delta$.
Differentiating \eqref{graph-system} once in $x$ gives
\begin{align*}
V_t
&=\partial_x(a(V)U_{xx})
-\delta\partial_x(xU_x)+\delta U_x\\
&=\partial_x(a(V)V_x)-\delta(V+xV_x)+\delta V\\
&=(a(V)V_x)_x-\delta xV_x.
\end{align*}
The zeroth-order terms cancel and the drift has a minus sign. Each component $V^\alpha$ satisfies the scalar divergence-form equation
\begin{equation}\label{first-derivative-divergence-equation}
(V^\alpha)_t-(a(V)(V^\alpha)_x)_x
+\delta x(V^\alpha)_x=0.
\end{equation}
Although the components determine the common coefficient $a(V)$, the scalar interior estimate uses only its ellipticity bounds and the bound for the drift; it does not require derivatives of $a$.

\smallskip
\noindent\emph{2. Uniform cylinders and the first Schauder estimate.}
Fix $(x_0,t_0)\in[-r,r]\times[\tau,S)$, and let
\[
R_*=\min\{r,\sqrt{\tau}/2\},\qquad
Q_\rho=(x_0-\rho,x_0+\rho)\times(t_0-\rho^2,t_0]
\quad(0<\rho\leq R_*).
\]
The cylinder $Q_{R_*}$ lies strictly inside the spatial domain and has
$t_0-R_*^2\geq3\tau/4>0$. Its size is independent of $t_0$ and $S$.

Interior De Giorgi--Nash--Moser estimates applied componentwise to
\eqref{first-derivative-divergence-equation} yield some $0<\alpha<1$ and a constant $C$ such that, for every component index $\beta$,
\[
\|V^\beta\|_{C^{\alpha,\alpha/2}(Q_{3R_*/4})}
\leq C\|V^\beta\|_{L^\infty(Q_{R_*})}\leq CM.
\]
The constants depend on the stated data, not on $t_0$ or $S$.

To see explicitly that the leading coefficient is Holder continuous, use
\[
|a(V(P))-a(V(Q))|
=\frac{|q(V(P))-q(V(Q))|}{q(V(P))q(V(Q))}
\leq C|V(P)-V(Q)|
\]
for any two space-time points $P,Q$ in this cylinder. Therefore
$\|a(V)\|_{C^{\alpha,\alpha/2}}\leq C$.
Now write each component of the original equation as
\[
(U^\beta)_t-a(V)(U^\beta)_{xx}
+\delta x(U^\beta)_x-\delta U^\beta=0.
\]
Interior Schauder estimates, with the drift and zeroth-order terms included as coefficients, give
\[
\|U^\beta\|_{C^{2+\alpha,1+\alpha/2}(Q_{R_*/2})}
\leq C\|U^\beta\|_{C^0(Q_{3R_*/4})}\leq CM.
\]
In particular $U_{xx}$ and $U_t$ are bounded and parabolically Holder continuous. This argument does not assume a time-Holder bound for $U$ before applying the estimate.

\smallskip
\noindent\emph{3. Higher spatial derivatives and the normalization terms.}
For $m\geq1$, set $W_m=D_x^mU$. The product rule gives
\[
D_x^m(xU_x)=xD_x^{m+1}U+mD_x^mU,
\qquad
D_x^m(U-xU_x)=(1-m)W_m-x(W_m)_x.
\]
Differentiating the leading term by Leibniz' rule, we obtain the exact equation
\begin{align*}
(W_m)_t
&=\sum_{j=0}^m\binom mj(D_x^ja)D_x^{m-j+2}U
-\delta x(W_m)_x+\delta(1-m)W_m\\
&=a(W_m)_{xx}+(m a_x-\delta x)(W_m)_x
+\delta(1-m)W_m+R_m,
\end{align*}
where
\begin{equation}\label{higher-derivative-remainder}
R_1=0,\qquad
R_m=\sum_{j=2}^m\binom mj(D_x^ja)D_x^{m-j+2}U
\quad(m\geq2).
\end{equation}
For example,
\begin{align*}
(W_1)_t&=a(W_1)_{xx}+(a_x-\delta x)(W_1)_x,\\
(W_2)_t&=a(W_2)_{xx}+(2a_x-\delta x)(W_2)_x
+(a_{xx}-\delta)W_2.
\end{align*}
The coefficient derivatives are also explicit:
\begin{align*}
a_x&=\sum_\beta a_{V^\beta}(U^\beta)_{xx},\\
a_{xx}
&=\sum_\beta a_{V^\beta}(U^\beta)_{xxx}
+\sum_{\beta,\gamma}a_{V^\beta V^\gamma}
(U^\beta)_{xx}(U^\gamma)_{xx}.
\end{align*}
More generally, $D_x^ja(V)$ is a finite sum of products of derivatives of $a$ with respect to $V$ and derivatives $D_x^hU$ of order at most $j+1$.

The estimate in part~2 makes $a_x$ Holder continuous. Apply Schauder estimates to the equation for $W_1$ on a smaller cylinder to bound
$W_1$ in $C^{2+\alpha,1+\alpha/2}$, obtaining control of $U_{xxx}$.
Inductively, suppose derivatives of $U$ through order $m+1$ are bounded and parabolically Holder continuous on the current cylinder.
Then $D_x^ja$ for $j\leq m$ and every term of $R_m$ are Holder bounded: in \eqref{higher-derivative-remainder}, the derivative of $U$ outside the coefficient has order $m-j+2\leq m$.
The scalar Schauder estimate for each component of $W_m$ therefore gives a
$C^{2+\alpha,1+\alpha/2}$ bound on a smaller cylinder, controlling the next derivative of $U$.

For each fixed target order, only finitely many nested cylinders are required. Choose their radii between $R_*/2$ and $R_*/4$ after the first estimate. This provides bounds at every $(x_0,t_0)$ above, with constants depending on the order and $R_*$ but not on $t_0$ or $S$.

\smallskip
\noindent\emph{4. Time derivatives and the initial-time restriction.}
Once all needed spatial derivatives have been bounded, the original equation gives $U_t$. The equation for $W_m$ similarly gives
$\partial_tD_x^mU$. If further time derivatives are required, differentiate again; for example,
\[
U_{tt}
=\bigl(D_Va(U_x)\cdot U_{xt}\bigr)U_{xx}
+a(U_x)(U_t)_{xx}-\delta x(U_t)_x+\delta U_t.
\]
Each quantity on the right is controlled by the spatial estimates and the first time-derivative equations. Repeating this argument controls arbitrary fixed mixed derivatives. As $(x_0,t_0)$ was arbitrary, the asserted uniform bounds hold on
$[-r,r]\times[\tau,S)$.

The positive time length of the backward cylinders is essential: bounds for $U$ and $U_x$ alone do not bound arbitrarily high initial derivatives. Estimates including $t=0$ must additionally depend on the corresponding smooth initial data.
\end{proof}

The preceding theorem supplies uniform regularity. The following homogeneous estimate additionally preserves a decay rate when applied on sliding time intervals.

\begin{lemma}[Rate-preserving interior estimate]\label{rate-preserving interior estimate}
Let $W:S^1\times[t_*,\infty)\to\mathbb R^N$ be a smooth solution of
\[
W_t=\alpha W_{\theta\theta}+B W_\theta+C W,
\]
where $\alpha$ is scalar, $0<\lambda\leq\alpha\leq\Lambda$, and $B,C$ are matrix-valued. Assume that these coefficients and all their spatial derivatives are uniformly bounded. For every integer $m\geq0$ there is a constant $C_m$, depending only on $m,N$, ellipticity, and finitely many coefficient bounds, such that
\begin{equation}\label{app:rate-preserving-estimate}
\|W(t)\|_{C^m(S^1)}
\leq C_m\left(\int_{t-1}^{t}\|W(s)\|_{L^2(S^1)}^2\,ds\right)^{1/2},
\qquad t\geq t_*+1.
\end{equation}
In particular, the constant is independent of $t$ and of the size of $W$. No symmetry of $B,C$ or bounds on time derivatives of the coefficients are needed.
\end{lemma}

\begin{proof}
We use the standard $L^2$ energy method. Set
\[
E_j(s)=\sum_{\ell=0}^j\|\partial_\theta^\ell W(s)\|_{L^2(S^1)}^2.
\]
We first derive the differential inequality for $E_j$. Write $W_\ell=\partial_\theta^\ell W$ and $\|\cdot\|_2=\|\cdot\|_{L^2(S^1)}$, using the fixed measure $d\theta$ and the Euclidean inner product on $\mathbb R^N$. Then
\[
\begin{aligned}
(W_\ell)_t
&=\partial_\theta^\ell\left(\alpha W_2+B W_1+C W_0\right)\\
&=\alpha W_{\ell+2}
  +\left(\ell\alpha_\theta I_N+B\right)W_{\ell+1}
  +\mathcal R_\ell,
\end{aligned}
\]
where $I_N$ is the identity matrix and
\[
\begin{aligned}
\mathcal R_\ell
={}&\sum_{r=2}^{\ell}\binom{\ell}{r}
       \left(\partial_\theta^r\alpha\right)W_{\ell-r+2}\\
&+\sum_{r=1}^{\ell}\binom{\ell}{r}
       \left(\partial_\theta^r B\right)W_{\ell-r+1}
 +\sum_{r=0}^{\ell}\binom{\ell}{r}
       \left(\partial_\theta^r C\right)W_{\ell-r},
\end{aligned}
\]
contains only low-order terms. So
\[
\|\mathcal R_\ell\|_2\leq K_\ell E_\ell^{1/2},
\]
where $K_\ell$ is independent of time and of $W$.

Taking the $L^2$ inner product of the differentiated equation with $W_\ell$, periodicity gives
\[
\int_{S^1}W_\ell\cdot\alpha W_{\ell+2}\,d\theta
=-\int_{S^1}\alpha|W_{\ell+1}|^2\,d\theta
 -\int_{S^1}\alpha_\theta W_\ell\cdot W_{\ell+1}\,d\theta.
\]
Consequently,
\[
\begin{aligned}
\frac12\frac{d}{dt}\|W_\ell\|_2^2
={}&-\int_{S^1}\alpha|W_{\ell+1}|^2\,d\theta\\
&+\int_{S^1}W_\ell\cdot
  \left(\left((\ell-1)\alpha_\theta I_N+B\right)W_{\ell+1}\right)\,d\theta
 +\int_{S^1}W_\ell\cdot\mathcal R_\ell\,d\theta.
\end{aligned}
\]
Let $P_\ell$ be a uniform bound for the operator norm of $(\ell-1)\alpha_\theta I_N+B$. Cauchy--Schwarz and Young's inequality imply
\[
\begin{aligned}
\left|\int_{S^1}W_\ell\cdot
  \left(\left((\ell-1)\alpha_\theta I_N+B\right)W_{\ell+1}\right)\,d\theta\right|
&\leq P_\ell\|W_\ell\|_2\|W_{\ell+1}\|_2\\
&\leq\frac{\lambda}{2}\|W_{\ell+1}\|_2^2
 +\frac{P_\ell^2}{2\lambda}\|W_\ell\|_2^2,\\
\left|\int_{S^1}W_\ell\cdot\mathcal R_\ell\,d\theta\right|
&\leq K_\ell\|W_\ell\|_2 E_\ell^{1/2}
\leq K_\ell E_\ell.
\end{aligned}
\]
Since $\alpha\geq\lambda$ and $\|W_\ell\|_2^2\leq E_\ell$, multiplying the resulting inequality by two yields
\[
\frac{d}{dt}\|W_\ell\|_2^2
+\lambda\|W_{\ell+1}\|_2^2
\leq\left(\frac{P_\ell^2}{\lambda}+2K_\ell\right)E_\ell.
\]
Summing over $\ell=0,\ldots,j$ and using $E_\ell\leq E_j$, we obtain
\[
E_j'+\lambda\sum_{\ell=0}^j\|W_{\ell+1}\|_2^2
\leq\left(\sum_{\ell=0}^j
  \left(\frac{P_\ell^2}{\lambda}+2K_\ell\right)\right)E_j.
\]
The sum on the left is $E_{j+1}-E_0$. Adding $\lambda E_0\leq\lambda E_j$ to the right-hand side therefore gives
\begin{equation}\label{app:derivative-energy}
E_j'+\lambda E_{j+1}\leq C_jE_j,
\end{equation}
where $C_j=\lambda+\sum_{\ell=0}^j\left(P_\ell^2/\lambda+2K_\ell\right)$.

Fix $t\geq t_*+1$ and arbitrarily choose $t-1\leq a<b<t$. Choose a smooth nondecreasing time cutoff function $\eta$ on $[a,t]$ such that
\[
0\leq\eta\leq1,\qquad \eta(a)=0,\qquad
\eta=1\text{ on }[b,t],\qquad
0\leq\eta'\leq\frac{2}{b-a}.
\]
Multiplying~\eqref{app:derivative-energy} by $\eta$ and using the product rule gives
\[
(\eta E_j)'+\lambda\eta E_{j+1}
\leq(C_j\eta+\eta')E_j.
\]
For any $\tau\in[b,t]$, integrate from $a$ to $\tau$. Since $\eta(a)=0$, $\eta(\tau)=1$, and $E_{j+1}\geq0$, we obtain
\[
\begin{aligned}
E_j(\tau)+\lambda\int_b^\tau E_{j+1}(s)\,ds
&\leq\int_a^\tau\left(C_j\eta(s)+\eta'(s)\right)E_j(s)\,ds\\
&\leq\left(C_j+\frac{2}{b-a}\right)\int_a^t E_j(s)\,ds.
\end{aligned}
\]
Discarding the nonnegative integral on the left and taking the supremum over $\tau\in[b,t]$ bounds $\sup_{[b,t]}E_j$ by the last expression. Taking instead $\tau=t$ and discarding $E_j(t)$ bounds $\int_b^t E_{j+1}$ by the same expression divided by $\lambda$. Adding these two bounds and absorbing the fixed factors into a constant independent of $a,b,t$ gives
\begin{equation}\label{app:cutoff-estimate}
\sup_{s\in[b,t]}E_j(s)+\int_b^t E_{j+1}(s)\,ds
\leq K_j\left(1+\frac{1}{b-a}\right)\int_a^t E_j(s)\,ds.
\end{equation}
Iterating for $j=0,\ldots,m+1$ on nested intervals, with $m+2$ equal gaps totaling $1/2$, yields
\[
\sup_{s\in[t-\frac12,t]}E_{m+1}(s)
\leq C_m\int_{t-1}^t E_0(s)\,ds.
\]
The coefficient bounds are uniform and the gaps depend only on $m$, so all constants are independent of $t$ and of the size of $W$. The one-dimensional embedding $H^{m+1}(S^1)\hookrightarrow C^m(S^1)$ proves~\eqref{app:rate-preserving-estimate}.
\end{proof}

\section{Evolution equations with tangential velocity}

Let $r(u,t)$ be a smooth closed strong spacelike curve in $\mathbb R^{n+2,k}$ evolving by
\begin{equation}\label{general-weighted-flow}
r_t=a r_{ss}+b r_s+\delta r,\qquad a>0,\quad\delta\in\{0,1\}.
\end{equation}
Here $u$ is a fixed periodic parameter, $s$ is the time-dependent pseudo-Euclidean arclength, and all $t$ derivatives are taken at fixed $u$. The Euclidean case is included by taking $k=0$. Set
\[
v=\sqrt{\langle r_u,r_u\rangle}>0,\qquad
ds=v\,du,\qquad \partial_s=v^{-1}\partial_u,
\]
and define
\[
T=r_s,\qquad \kappa=\sqrt{\langle r_{ss},r_{ss}\rangle}>0,\qquad
N=\frac{r_{ss}}{\kappa},\qquad \mathcal B=N_s+\kappa T.
\]
Since $\langle T,T\rangle=1$, differentiating in $s$ gives
$\langle T,T_s\rangle=0$, so $\langle T,N\rangle=0$ and
$\langle N,N\rangle=1$. Differentiating these identities once more yields
\[
\langle N_s,T\rangle=-\langle N,T_s\rangle=-\kappa,\qquad
\langle N_s,N\rangle=0.
\]
It follows that
\[
T_s=\kappa N,\qquad N_s=-\kappa T+\mathcal B,\qquad
\langle\mathcal B,T\rangle=\langle\mathcal B,N\rangle=0.
\]
A further differentiation of $\langle\mathcal B,N\rangle=0$ gives the identity needed below:
\begin{equation}\label{torsion-normal-identity}
\langle\mathcal B_s,N\rangle
=-\langle\mathcal B,N_s\rangle
=-\langle\mathcal B,-\kappa T+\mathcal B\rangle
=-\langle\mathcal B,\mathcal B\rangle.
\end{equation}

\begin{lemma}\label{derivative exchange of s and t}
For \eqref{general-weighted-flow},
\begin{equation}\label{metric-and-commutator}
\partial_t ds=(b_s+\delta-a\kappa^2)\,ds,\qquad
[\partial_t,\partial_s]=(a\kappa^2-b_s-\delta)\partial_s.
\end{equation}
\end{lemma}
\begin{proof}
Differentiate the metric density using the fixed parameter $u$:
\begin{align*}
v_t
&=\frac{1}{2v}\partial_t\langle r_u,r_u\rangle
=\frac1v\langle r_u,(r_t)_u\rangle\\
&=v\langle T,\partial_s(a\kappa N+bT+\delta r)\rangle.
\end{align*}
The spatial derivative of the velocity is
\begin{align*}
\partial_s(a\kappa N+bT+\delta r)
&=(a\kappa)_sN+a\kappa(-\kappa T+\mathcal B)
+b_sT+b\kappa N+\delta T\\
&=(-a\kappa^2+b_s+\delta)T
+\bigl((a\kappa)_s+b\kappa\bigr)N+a\kappa\mathcal B.
\end{align*}
Taking its inner product with $T$ therefore gives
\[
v_t=(b_s+\delta-a\kappa^2)v.
\]
Because $ds=v\,du$ with $u$ fixed, this proves the metric formula.

For any smooth scalar function $\psi(u,t)$, use
$\partial_s\psi=v^{-1}\psi_u$ to compute
\begin{align*}
\partial_t(\partial_s\psi)
&=\partial_t(v^{-1}\psi_u)
=-v^{-2}v_t\psi_u+v^{-1}(\psi_t)_u\\
&=(a\kappa^2-b_s-\delta)\partial_s\psi
+\partial_s(\partial_t\psi).
\end{align*}
This is the asserted commutator; it also applies componentwise to vector-valued functions.
\end{proof}

We record the associated integral differentiation rule. For any smooth $\psi$ on the closed parameter domain,
\begin{align}
\frac{d}{dt}\int_{r(t)}\psi\,ds
&=\frac{d}{dt}\int\psi(u,t)v(u,t)\,du \notag\\
&=\int_{r(t)}
\bigl(\psi_t+(b_s+\delta-a\kappa^2)\psi\bigr)\,ds.
\label{moving-curve-integral}
\end{align}

\begin{remark}
The metric and commutator do not require strong spacelikeness.
For a merely spacelike curve let $\mathcal H=T_s=r_{ss}$, without normalizing this vector. Then
\[
\langle T,\mathcal H\rangle=0,\qquad
\langle T,\mathcal H_s\rangle
=-\langle T_s,\mathcal H\rangle
=-\langle\mathcal H,\mathcal H\rangle.
\]
Consequently, the same metric calculation becomes
\begin{align*}
v_t/v
&=\langle T,\partial_s(a\mathcal H+bT+\delta r)\rangle\\
&=a_s\langle T,\mathcal H\rangle
+a\langle T,\mathcal H_s\rangle+b_s
+b\langle T,\mathcal H\rangle+\delta\\
&=b_s+\delta-a\langle\mathcal H,\mathcal H\rangle.
\end{align*}
Thus \eqref{metric-and-commutator} holds with $\kappa^2$ replaced by
$\langle\mathcal H,\mathcal H\rangle$, which may be negative or zero.
In particular, adding $b r_s$ changes the local commutator by $-b_s\partial_s$; this term cannot be omitted.
\end{remark}

\begin{proposition}\label{curvature evolution}
Under \eqref{general-weighted-flow},
\begin{equation}\label{weighted-curvature-evolution}
\kappa_t=(a\kappa)_{ss}+b\kappa_s+a\kappa^3
-a\kappa\langle\mathcal B,\mathcal B\rangle-\delta\kappa.
\end{equation}
\end{proposition}
\begin{proof}
First apply the commutator to $T=r_s$:
\begin{align*}
T_t
&=\partial_s r_t+(a\kappa^2-b_s-\delta)r_s\\
&=\partial_s(a\kappa N+bT+\delta r)
+(a\kappa^2-b_s-\delta)T\\
&=\bigl((a\kappa)_s+b\kappa\bigr)N+a\kappa\mathcal B.
\end{align*}
The $T$ components cancel exactly. Abbreviate
$\alpha=(a\kappa)_s+b\kappa$. Differentiating this expression in $s$ gives
\begin{align*}
(T_t)_s
&=\alpha_sN+\alpha N_s+(a\kappa)_s\mathcal B+a\kappa\mathcal B_s\\
&=-\kappa\alpha T+\alpha_sN
+\bigl(\alpha+(a\kappa)_s\bigr)\mathcal B+a\kappa\mathcal B_s.
\end{align*}
Since $\kappa^2=\langle T_s,T_s\rangle$ and $\kappa>0$,
\begin{align*}
\kappa_t
&=\frac1{2\kappa}\partial_t\langle T_s,T_s\rangle
=\frac1\kappa\langle T_s,(T_s)_t\rangle\\
&=\left\langle N,(T_t)_s+
(a\kappa^2-b_s-\delta)T_s\right\rangle\\
&=\alpha_s+a\kappa\langle N,\mathcal B_s\rangle
+(a\kappa^2-b_s-\delta)\kappa.
\end{align*}
Now use \eqref{torsion-normal-identity} and
$\alpha_s=(a\kappa)_{ss}+b_s\kappa+b\kappa_s$:
\begin{align*}
\kappa_t
&=(a\kappa)_{ss}+b_s\kappa+b\kappa_s
-a\kappa\langle\mathcal B,\mathcal B\rangle
+a\kappa^3-b_s\kappa-\delta\kappa\\
&=(a\kappa)_{ss}+b\kappa_s+a\kappa^3
-a\kappa\langle\mathcal B,\mathcal B\rangle-\delta\kappa.
\end{align*}
This also exhibits the cancellation of the two $b_s\kappa$ terms.
\end{proof}

\begin{proposition}\label{total curvature evolution}
For the closed curves in \eqref{general-weighted-flow}, their total curvature satisfies
\[
K(t)=\int_{r(t)}\kappa\,ds,\qquad
K'(t)=-\int_{r(t)}a\kappa\langle\mathcal B,\mathcal B\rangle\,ds.
\]
\end{proposition}
\begin{proof}
Substitute $\psi=\kappa$ in \eqref{moving-curve-integral} and then use
\eqref{weighted-curvature-evolution}:
\begin{align*}
K'
&=\int_{r(t)}\bigl(\kappa_t+
(b_s+\delta-a\kappa^2)\kappa\bigr)\,ds\\
&=\int_{r(t)}\bigl((a\kappa)_{ss}+b\kappa_s
+a\kappa^3-a\kappa\langle\mathcal B,\mathcal B\rangle-\delta\kappa\\
&\hspace{42mm}{}+b_s\kappa+\delta\kappa-a\kappa^3\bigr)\,ds\\
&=\int_{r(t)}\bigl((a\kappa)_{ss}+(b\kappa)_s
-a\kappa\langle\mathcal B,\mathcal B\rangle\bigr)\,ds.
\end{align*}
At a fixed time, parametrize once around the full parameter domain by arclength $s\in[0,L(t)]$. Periodicity gives
\[
\int_0^{L(t)}(a\kappa)_{ss}\,ds
=[(a\kappa)_s]_0^{L(t)}=0,\qquad
\int_0^{L(t)}(b\kappa)_s\,ds=[b\kappa]_0^{L(t)}=0.
\]
This proves the formula. The argument uses periodicity of the parametrized curve, not injectivity, so it remains valid for closed immersions and multiple coverings.
\end{proof}

\smallskip
\noindent\emph{Two sign conventions and their consequences.}
In Euclidean space,
$\tau^2=|\mathcal B|_E^2=\langle\mathcal B,\mathcal B\rangle\geq0$,
and the formulas read
\[
\kappa_t=(a\kappa)_{ss}+b\kappa_s+a\kappa^3-a\kappa\tau^2-\delta\kappa,
\qquad K'=-\int a\kappa\tau^2\,ds\leq0.
\]
In $\mathbb R^{2,k}$, the orthogonal complement of the positive definite plane
$\operatorname{span}\{T,N\}$ is negative definite. Thus $\mathcal B$ is timelike or zero. Defining
$\tau^2=-\langle\mathcal B,\mathcal B\rangle\geq0$ gives
\[
\kappa_t=(a\kappa)_{ss}+b\kappa_s+a\kappa^3+a\kappa\tau^2-\delta\kappa,
\qquad K'=\int a\kappa\tau^2\,ds\geq0.
\]
In particular, for unnormalized CSF ($a=1$, $b=\delta=0$) in $\mathbb R^{2,1}$,
\[
\kappa_t=\kappa_{ss}+\kappa^3+\kappa\tau^2,\qquad
K'=\int\kappa\tau^2\,ds.
\]
At a spatial minimum of $\kappa$, the first equation gives
$\kappa_t\geq\kappa^3\geq0$. On a smooth spacelike time interval starting from a closed strong spacelike curve, the maximum principle therefore yields
$\min\kappa(\cdot,t)\geq\min\kappa(\cdot,0)>0$ as long as $\kappa$ is defined.
If strong spacelikeness were first lost at a smooth spacelike time, continuity would force $\langle r_{ss},r_{ss}\rangle=\kappa^2$ to reach zero, contradicting this bound. This is the preservation argument used for the clam-shell example.

\section{Exponential convergence}\label{sec:exponential-convergence}

We now improve the smooth convergence proved in Section~\ref{sec:isoperimetric-convergence} to an explicit exponential rate. The argument uses polar graphs, two energy estimates, and the rate-preserving interior estimate of Lemma~\ref{rate-preserving interior estimate}. All norms and all inner products in the energies below are taken in fixed auxiliary Euclidean coordinates; the indefinite transverse product is always indicated explicitly. Constants may depend on the given flow and the derivative order, but not on the late normalized time.

\begin{theorem}[Exponential round-point convergence]\label{exponential convergence}
For the normalized flow in the round-point case established in Sections~\ref{sec:tangent-estimates}--\ref{sec:isoperimetric-convergence}, let
\[
C_\infty(\theta)=v_1\cos\theta+v_2\sin\theta,
\qquad \langle v_i,v_j\rangle=\delta_{ij},
\]
be its limiting single-covered centered unit circle. There are smooth diffeomorphisms $\varphi_\sigma:S^1\to S^1$ and constants $C_m$ such that
\begin{equation}\label{exp:curve-rate}
\|\gamma(\varphi_\sigma(\cdot),\sigma)-C_\infty\|_{C^m(S^1;E)}
\leq C_m e^{-2\sigma},\qquad m\geq0,
\end{equation}
on a sufficiently late time interval. The circle and its plane are fixed in time.
\end{theorem}

\begin{proof}
\emph{1. Polar equations and the interior estimate.}
Work in the fixed coordinates chosen at the end of Section~\ref{sec:planar-convergence}, so that Section~\ref{sec:isoperimetric-convergence} gives the limit $C(\theta)=(\cos\theta,\sin\theta,0)$. In this section, we change the notion of $\theta$ to the common polar parameter on the $xy$-plane and write
\[
\gamma=(r e_r,Z),\qquad e_r=(\cos\theta,\sin\theta),\qquad
r=1+u>0,\qquad Z\in\mathbb R^{n,k}.
\]
We derive the evolution equations of $u=r-1$ and $Z$ under the polar parameterization. Let
\begin{equation}\label{exp:graph-data}
\begin{gathered}
J=\operatorname{diag}(I_n,-I_k),\qquad
\langle z,w\rangle_J=z^{\mathsf T}Jw,\qquad h=Z/r,\\
g=\langle\gamma_\theta,\gamma_\theta\rangle
=r^2+r_\theta^2+\langle Z_\theta,Z_\theta\rangle_J,
\qquad \alpha=g^{-1}.
\end{gathered}
\end{equation}
Here $u$ denotes the radial error only in this section, and $\alpha$ is the diffusion coefficient, not the extinction point $a$. The previously established smooth convergence and Corollary~\ref{coordinate derivative bound} imply
\[
u,Z\to0,\qquad \alpha\to1
\quad\hbox{in every spatial }C^m\hbox{ norm}.
\]
In particular all coefficient derivatives are uniformly bounded, and we may assume $r\geq1/2$ and $2/3\leq\alpha\leq2$.

Since $\partial_s=g^{-1/2}\partial_\theta$,
\[
\gamma_{ss}=\alpha\gamma_{\theta\theta}
-\frac{g_\theta}{2g^2}\gamma_\theta.
\]
The tangential correction
\[
\beta=\frac{g_\theta}{2g^2}-\frac{2r_\theta}{gr}
\]
in $\gamma_\sigma=\gamma_{ss}+\gamma+\beta\gamma_\theta$ cancels its angular component. The radial and transverse components are therefore
\begin{align}
r_\sigma&=\alpha\left(r_{\theta\theta}-r-\frac{2r_\theta^2}{r}\right)+r,
\label{exp:r-equation}\\
Z_\sigma&=\alpha\left(Z_{\theta\theta}-\frac{2r_\theta}{r}Z_\theta\right)+Z.
\label{exp:Z-equation}
\end{align}
Substituting $Z=rh$ and subtracting~\eqref{exp:r-equation} multiplied by $h$ cancels the mixed first derivatives and gives
\begin{equation}\label{exp:h-equation}
h_\sigma=\alpha L_1h,\qquad L_1=\partial_\theta^2+1.
\end{equation}
Similarly, $g-1=2u+u^2+u_\theta^2+\langle Z_\theta,Z_\theta\rangle_J$ yields
\begin{equation}\label{exp:u-equation}
u_\sigma=\alpha L_2u+R_r+R_z,\qquad L_2=\partial_\theta^2+2,
\end{equation}
where the exact remainders are
\begin{equation}\label{exp:remainders}
R_r=\alpha\left[3u^2+u^3+\left(r-\frac2r\right)u_\theta^2\right],
\qquad R_z=\alpha r\langle Z_\theta,Z_\theta\rangle_J.
\end{equation}

We record why Lemma~\ref{rate-preserving interior estimate} applies without an independent source term. Along the given solution, $U=(u,Z)$ satisfies
\[
U_\sigma=\alpha U_{\theta\theta}+BU_\theta+CU.
\]
If $\epsilon_j=J_{jj}$, the nonzero matrix entries are
\begin{align*}
B_{00}&=\alpha(r-2/r)u_\theta,& B_{0j}&=\alpha r\epsilon_jZ_\theta^j,\\
B_{jj}&=-2\alpha u_\theta/r,& C_{00}&=\alpha r(r+1),\qquad C_{jj}=1,
\end{align*}
for $1\leq j\leq n+k$. Every unspecified entry is zero. These coefficients have uniform spatial derivative bounds. Thus the cited lemma gives
\begin{equation}\label{exp:U-smoothing}
\|U(\sigma)\|_{C^m}
\leq C_m\left(\int_{\sigma-1}^{\sigma}\|U(s)\|_{L^2}^2\,ds\right)^{1/2}.
\end{equation}

\smallskip
\emph{2. Transverse decay.}
Set
\[
\mathcal Q=\int_{S^1}(|h_\theta|_E^2-|h|_E^2)\,d\theta,
\qquad v=L_1h.
\]
Periodic integration by parts in~\eqref{exp:h-equation} gives
\begin{equation}\label{exp:Q-dissipation}
\mathcal Q'=-2\int_{S^1}\alpha|v|_E^2\,d\theta.
\end{equation}
Although this quadratic form is not nonnegative on all functions, $h\to0$ gives $\mathcal Q(\infty)=0$. Its monotonicity therefore implies $\mathcal Q\geq0$ along the present solution. Parseval's identity gives, with $h_j$ the Fourier coefficients,
\[
\|v\|_{L^2}^2-3\mathcal Q
=2\pi\sum_{j\in\mathbb Z}(j^2-1)(j^2-4)|h_j|_E^2\geq0.
\]
Every displayed multiplier is nonnegative for integer $j$, so no low mode has been discarded. Since $\alpha\geq2/3$,~\eqref{exp:Q-dissipation} implies
\begin{equation}\label{exp:transverse-energy-rate}
0\leq\mathcal Q(\sigma)\leq Ce^{-4\sigma},\qquad
\int_{\sigma-1}^{\sigma}\|v(s)\|_{L^2}^2\,ds\leq Ce^{-4\sigma}.
\end{equation}
The equation
\[
v_\sigma=L_1(\alpha v)
=\alpha v_{\theta\theta}+2\alpha_\theta v_\theta
+(\alpha_{\theta\theta}+\alpha)v
\]
has uniformly bounded coefficients of all spatial orders. Lemma~\ref{rate-preserving interior estimate} gives $\|v\|_{C^m}\leq C_m e^{-2\sigma}$. Since $h(\infty)=0$,~\eqref{exp:h-equation} now yields the convergent terminal integral
\[
h(\sigma)=-\int_\sigma^\infty\alpha(s)v(s)\,ds.
\]
The product rule and $Z=rh$ consequently give
\begin{equation}\label{exp:transverse-rate}
\|h(\sigma)\|_{C^m}+\|Z(\sigma)\|_{C^m}\leq C_m e^{-2\sigma}.
\end{equation}
In particular the terminal integral controls the first harmonics in $\ker L_1$; their coefficients need not vanish identically. When $n+k=0$, the transverse quantities are absent and~\eqref{exp:transverse-rate} is vacuous.

\smallskip
\emph{3. A preliminary radial rate.}
Define
\[
\mathcal E=\int_{S^1}(u_\theta^2-2u^2)\,d\theta,
\qquad \mathcal D=\|L_2u\|_{L^2}^2.
\]
Again by Parseval,
\begin{equation}\label{exp:radial-coercivity}
\begin{gathered}
\mathcal D-2\mathcal E
=2\pi\sum_{j\in\mathbb Z}(j^2-2)(j^2-4)|u_j|^2\geq0,\\
\|u\|_{H^2}\leq C\mathcal D^{1/2}.
\end{gathered}
\end{equation}
The last estimate follows because $L_2$ has no zero Fourier eigenvalue. Neither estimate requires $\mathcal E\geq0$ or any orthogonality condition on $u$. Differentiation and~\eqref{exp:u-equation} give
\begin{equation}\label{exp:radial-energy-identity}
\mathcal E'=-2\int\alpha(L_2u)^2\,d\theta
-2\int(L_2u)(R_r+R_z)\,d\theta.
\end{equation}
Equations~\eqref{exp:remainders},~\eqref{exp:transverse-rate}, and~\eqref{exp:radial-coercivity} imply
\[
\|R_r\|_{L^2}\leq C\|u\|_{C^1}\mathcal D^{1/2},\qquad
\|R_z\|_{L^2}\leq C\|Z_\theta\|_\infty\|Z_\theta\|_{L^2}
\leq Ce^{-4\sigma}.
\]
Using $\|u\|_{C^1}\to0$, $\alpha\to1$, and Young's inequality in~\eqref{exp:radial-energy-identity}, we obtain on a later tail
\begin{equation}\label{exp:coarse-energy-inequality}
\mathcal E'\leq-\mathcal D+C_8e^{-8\sigma}.
\end{equation}
Put
\[
H_8(\sigma)=\int_\sigma^\infty C_8e^{-8s}\,ds,
\qquad \mathcal E_8=\mathcal E+H_8.
\]
Then $\mathcal E_8'\leq-\mathcal D$ and $\mathcal E_8(\infty)=0$, so $\mathcal E_8\geq0$. By~\eqref{exp:radial-coercivity},
\[
\mathcal E_8'\leq-2\mathcal E_8+2H_8.
\]
Gronwall's inequality and the dissipation estimate give
\[
\mathcal E_8(\sigma)\leq Ce^{-2\sigma},\qquad
\int_{\sigma-1}^{\sigma}\mathcal D(s)\,ds
\leq\mathcal E_8(\sigma-1)\leq Ce^{-2\sigma}.
\]
Combining this with~\eqref{exp:radial-coercivity},~\eqref{exp:transverse-rate}, and the already established interior estimate~\eqref{exp:U-smoothing} yields
\begin{equation}\label{exp:preliminary-rate}
\|u(\sigma)\|_{C^m}\leq C_m e^{-\sigma}.
\end{equation}

\smallskip
\emph{4. The endpoint exponent.}
Equation~\eqref{exp:graph-data} and the preceding rates imply
\[
\|\alpha-1\|_\infty\leq Ce^{-\sigma},\qquad
\|R_r\|_{L^2}\leq Ce^{-\sigma}\mathcal D^{1/2},\qquad
\|R_z\|_{L^2}\leq Ce^{-4\sigma}.
\]
Apply~\eqref{exp:radial-energy-identity} once more. This time choose a time-dependent Young parameter:
\[
Ce^{-4\sigma}\mathcal D^{1/2}
\leq e^{-\sigma}\mathcal D+C'e^{-7\sigma}.
\]
Consequently, for fixed constants $C_0,C_7$,
\begin{equation}\label{exp:sharp-energy-inequality}
\mathcal E'\leq-(2-C_0e^{-\sigma})\mathcal D+C_7e^{-7\sigma}.
\end{equation}
As in Step~3, define
\[
H_7(\sigma)=\int_\sigma^\infty C_7e^{-7s}\,ds,
\qquad \mathcal E_7=\mathcal E+H_7.
\]
On a tail where $2-C_0e^{-\sigma}\geq1$, its dissipation and zero terminal limit give $\mathcal E_7\geq0$. Using~\eqref{exp:radial-coercivity} now gives
\begin{equation}\label{exp:sharp-energy-ode}
\mathcal E_7'\leq-(4-2C_0e^{-\sigma})\mathcal E_7+Ce^{-7\sigma}.
\end{equation}
The coefficient error is integrable, and for $\sigma\geq s\geq\sigma_0$,
\[
\exp\left(-\int_s^\sigma(4-2C_0e^{-\xi})\,d\xi\right)
\leq C e^{-4(\sigma-s)}.
\]
Gronwall's inequality applied to~\eqref{exp:sharp-energy-ode} therefore yields
\[
\mathcal E_7(\sigma)\leq Ce^{-4\sigma},\qquad
\int_{\sigma-1}^{\sigma}\mathcal D(s)\,ds\leq Ce^{-4\sigma}.
\]
The second estimate follows by integrating~\eqref{exp:sharp-energy-inequality} after adding $H_7$. By~\eqref{exp:radial-coercivity},~\eqref{exp:transverse-rate}, and~\eqref{exp:U-smoothing},
\begin{equation}\label{exp:final-graph-rate}
\|u(\sigma)\|_{C^m}+\|Z(\sigma)\|_{C^m}\leq C_m e^{-2\sigma}.
\end{equation}
Finally $\gamma-C=(u e_r,Z)$, and all derivatives of $e_r$ are bounded. The estimate follows in the polar gauge. Returning to the original fixed coordinates changes only the constants and proves~\eqref{exp:curve-rate}, hence~\eqref{main-exponential-rate}.
\end{proof}

\begin{corollary}\label{exponential rate in original time}
In the round-point case, with the extinction point $a$ and the reparametrizations of Theorem~\ref{exponential convergence},
\begin{equation}\label{exp:original-time-rate}
\|\Gamma(\varphi_t(\cdot),t)-a-\sqrt{2(T-t)}\,C_\infty\|_{C^m(S^1;E)}
\leq\widetilde C_m(T-t)^{3/2}
\end{equation}
for every $m\geq0$ and all sufficiently late $t<T$, where $\varphi_t=\varphi_{\sigma(t)}$.
\end{corollary}
\begin{proof}
Use the CSF/NF correspondence~\eqref{expression of NF} with $T_0=T$ after translating by $a$. Since $e^{-2\sigma}=(T-t)/T$, multiplying~\eqref{exp:curve-rate} by $\sqrt{2(T-t)}$ gives~\eqref{exp:original-time-rate}.
\end{proof}

\begin{remark}\label{exp:transverse-improvement}
In the fixed limiting-plane gauge, the transverse estimate improves further to $\|Z(\sigma)\|_{C^m}\leq C_m e^{-3\sigma}$. Indeed,~\eqref{exp:final-graph-rate} gives $\|\alpha-1\|_\infty\leq Ce^{-2\sigma}$, and~\eqref{exp:Q-dissipation} then implies $\mathcal Q'\leq-6(1-Ce^{-2\sigma})\mathcal Q$. The same integrable-error comparison and Step~2 give the claim.
\end{remark}

\bibliographystyle{alpha}
\bibliography{2}
\end{document}